\documentclass[10pt, a4paper, twoside]{amsart}

\usepackage[english]{babel} 
\usepackage[utf8]{inputenc}
\usepackage[T1]{fontenc}
\usepackage[shortlabels]{enumitem}
\usepackage{comment}
\usepackage{appendix}
\usepackage{latexsym,bm,lmodern,xcolor,graphicx,setspace,pdfpages,
geometry,marginnote,caption,array,siunitx,eurosym,colortbl,float,wrapfig,vmargin,tikz,hyperref}
\usepackage{mathtools, mathabx, amssymb, commath}

\setmarginsrb{30mm}{20mm}{30mm}{30mm}   
                 {8mm}{6mm}{10mm}{10mm} 
\definecolor{Viola}{RGB}{188,30,188}
\definecolor{Blu}{RGB}{0, 44, 92}
\hypersetup{colorlinks=true,linkcolor=Blu,allbordercolors=white, urlcolor=blue, citecolor=Viola}

\theoremstyle{plain}
\newtheorem{theorem}{Theorem}[section]
\newtheorem{lemma}[theorem]{Lemma}
\newtheorem{proposition}[theorem]{Proposition}
\newtheorem{corollary}[theorem]{Corollary}
\newtheorem{conjecture}{Conjecture}
\theoremstyle{definition}
\newtheorem{definition}[theorem]{Definition}
\newtheorem{example}[theorem]{Example}
\theoremstyle{remark}
\newtheorem{remark}[theorem]{Remark}

\DeclareMathOperator{\upper}{\mathbb{H}}
\DeclareMathOperator{\uppertwo}{\mathbb{H}^2}       
\DeclareMathOperator{\uppern}{\mathbb{H}^{\mathit{d}}}
\DeclareMathOperator{\upperd}{\mathbb{H}^{\mathit{d}}}
\DeclareMathOperator{\R}{\mathbb{R}}
\DeclareMathOperator{\Rtwo}{\mathbb{R}^2}
\DeclareMathOperator{\Rn}{\mathbb{R}^{\mathit{d}}}
\DeclareMathOperator{\Rd}{\mathbb{R}^{\mathit{d}}}
\DeclareMathOperator{\C}{\mathbb{C}}

\DeclareMathOperator{\im}{Im}
\DeclareMathOperator{\re}{Re}
\DeclareMathOperator{\supp}{supp}
\DeclareMathOperator{\D}{\mathbb{D}}

\DeclareMathOperator{\tohat}{\hat{\to}}
\DeclareMathOperator{\Ptwo}{\mathcal{P}_2}
\DeclareMathOperator{\Pn}{\mathcal{P}_{\mathit{d}}}
\DeclareMathOperator{\Z}{\mathcal{Z}}

\def\HNf{Herglotz-Nevanlinna function}

\title[Real rational HN-functions and extremals]{Real rational Herglotz-Nevanlinna functions in two variables and a class of extremals }

\author{Ludvig Fagrell}
\address{Independent Researcher}
\email{fagrell.ludvig@gmail.com}

\author{Annemarie Luger}
\address{Department of Mathematics, Stockholm University, 106 91 Stockholm, Sweden}
\email{luger@math.su.se}

\date{October 1, 2026}

\begin{document}

\begin{abstract}
We first revisit the Nevanlinna condition for representing measures of Herglotz-Nevanlinna functions in two variables and give a general operator-theoretic criterion for extremality. We then derive an explicit formula for the representing measure of a real rational \HNf{} in two variables, and combine it with the extremality criterion to obtain a family of extreme elements in the full cone of Nevanlinna measures. The family includes examples with reducible denominators, and these yield counterexamples to the irreducibility assertion in a conjecture of Knese. We then characterize reducibility within this family both algebraically and geometrically.
\end{abstract}
\maketitle

\section{\bfseries Introduction}

A holomorphic function $h:\uppern\to\C$ is called a \emph{\HNf} if its imaginary part is nonnegative. Whereas for $d=1$ this is a very well studied class, in the case $d \geq 2$ much less is known. In \cite{Luger_2019} a result for the polydisk, see \cite{KoranyiPukanszky1963}, has been used in order to show that such functions admit an integral representation in terms of constants $a\in\R$, $b\in [0,\infty)^d$ and a 
positive Borel measure on $\Rd$ satisfying a natural growth condition and an additional constraint called the \emph{Nevanlinna condition}. Lately these representing measures have been referred to as \emph{Nevanlinna measures} \cite{ Luger_2020_geometry_nevanlinna_measure, Nedic2020, nedic2022structurenevanlinnameasures}.

Nevanlinna measures are still not very well understood  since the defining condition describes them  only implicitly as the solution set of infinitely many equations. In \cite{nedic2022structurenevanlinnameasures}, the authors obtained an alternative Fourier transform characterization of Nevanlinna measures. In case $d=2$, yet another way to view Nevanlinna measures is proposed in the current text. To this end we introduce an integral operator $\mathcal{T}_\mu$ associated with any two-dimensional sufficiently bounded Borel measure $\mu$. The Nevanlinna condition then reads as $\mathcal{T}_\mu1=0$, cf., Subsection \ref{section: extremality criterion}. 

The set of Nevanlinna measures on $\Rd$ forms a convex cone in the real vector space of signed Borel measures on $\Rd$. It is therefore natural to ask which measures are its extremal elements. Equivalently, one may ask when the corresponding \HNf{} has a nontrivial decomposition into \HNf s. In the case $d=1$, it is well known that the extremal measures are precisely the point measures. These correspond to (positive multiples of) functions of the form $1/(\tau-z)$, where $\tau$ is a real number. For $d>1$, however, the Nevanlinna condition makes the geometry of the cone substantially more complicated, and presently the problem of extremality remains largely unresolved; only partial results exist. For example, extremal functions in one variable are also extremal when lifted to several variables, and for such a function the representing measure is a product measure of a point measure and a Lebesgue-measure on $\R^{d-1}$. See \cite{Nedic2020} for a characterization of general product Nevanlinna measures. The set of extremal measures is obviously much larger than the collection of extremal product measures, not least because extremal Nevanlinna measures can be absolutely continuous \cite{acRPmeasure}.

Additional partial results on the extremality problem have been obtained for some particularly rigid subclasses of Nevanlinna measures. In \cite{Luger_2020_geometry_nevanlinna_measure},  restrictions of Nevanlinna measures to affine subspaces are studied, and it was shown that coordinate orthogonal hyperplane components yield natural decompositions. In \cite{nedic2022structurenevanlinnameasures} an explicit description of Nevanlinna measures supported on hyperplanes is obtained. Within this class, the extremal measures were determined as those with densities which are squares of affine functions. 

Our contribution here is an operator-theoretic criterion for extremality of $\mu$ in terms of the kernel of $\mathcal{T}_\mu$. Specifically, $\mu$ is extremal precisely when $\ker\mathcal{T}_\mu=\mathbb{R}\mathbf{1}$. Part of the underlying argument is reminiscent of the perturbation of polynomials discussed in \cite{KneseConjecturePaper} for rational inner functions. 

We also want to highlight results  for the related class of analytic functions from the polydisk $\mathbb D^d$ to the right half-plane. In particular, we want to mention the foundational works by McDonald \cite{McDonald1,  ExamplesOfRPMeasures, McDonald2,acRPmeasure} as well as more recent results including \cite{clark_measures_two, clark_measures_one} and \cite{KneseConjecturePaper}.

The current paper then specializes to the important subclass of real rational Herglotz-Nevanlinna functions, i.e., rational Herglotz-Nevanlinna functions which (except for the singularities) are real-valued on $\mathbb R^d$.
 There is a closely related theory in the polydisc where the parallel objects are \emph{rational inner functions} and their corresponding \emph{Clark measures}. Explicit descriptions of Clark measures for rational inner functions on the bidisc were obtained in \cite{clark_measures_one} for bidegree $(m,1)$ and extended by Anderson et al.~\cite{clark_measures_two} to general bidegrees. In particular, \cite{clark_measures_two} expresses the Clark measure in terms of analytic functions that parametrize level sets of the underlying rational inner function. More recently, Calzi \cite{newClarkStudy} obtained a similar description of Clark measures for rational inner functions on general bounded symmetric domains. Since the poly-upper half-plane and polydisc settings are biholomorphically equivalent, an analogous description of Nevanlinna measures for real rational Herglotz-Nevanlinna functions in two variables has to exist. Results may in principle be transferred between them \cite[Section 2.3]{nedic2022structurenevanlinnameasures}, however, this transformation is not completely transparent. We therefore derive the analogous structure formula directly from the Herglotz-Nevanlinna integral representation. Specifically, the representing measure of a real rational \HNf{} in two variables can be written as a sum of terms of the form
\begin{equation*}
     (\gamma)_*(\alpha\lambda|_{\R\setminus S}) \quad\text{and}\quad  c_{z}(\tau)(\delta_{\tau}\times\lambda|_{\R\setminus S}),
\end{equation*}
where $\lambda$ denotes the Lebesgue measure, $\gamma$ is the graph of a function on $\R\setminus S$ which parametrizes (part of) the real zero set of the denominator, $S\subset\R$ is a finite exceptional set, $\alpha$ is a weight function which is determined by the numerator and denominator, and $c_z(\tau)$ is a nonnegative constant obtained from a particular limit of the function. An easy consequence of the structure formula is the identification of the support of the measure with the real zero set of the denominator.

From here, we consider the extremality problem for the subclass of real rational Herglotz-Nevanlinna functions in two variables. 
Using the structure formula we refine the operator-theoretic criterion for extremality in a way that allows us to express the problem in terms of the algebraic information encoded in the corresponding function. This enables us  to identify certain subclasses of extreme functions. 

As a first observation it can be mentioned that for real rational Herglotz-Nevanlinna functions of total degree 2 in the denominator it is straight forward to characterize the extremal functions. In particular,  in this subclass the Herglotz-Nevanlinna property forces  in  all but one case the denominator to factorize and, moreover, whenever the denominator factors, the function is not extremal. It is then natural to expect that irreducibility of the denominator would be a necessary condition for extremality. This was also conjectured in the polydisk case by Knese in \cite{KneseConjecturePaper}. But it turns out to be false, as illustrated in Subsection \ref{section: extremal_example} by the simple example 
\begin{equation}\label{eq: func in intro}
    (z,w)\mapsto -\frac{zw}{(z+w)(zw-1)}.
\end{equation}
The structure formula gives an explicit description of its representing measure, which is supported on the two real algebraic curves $w=-z$ and $w=1/z$ that make the denominator's real zero set. Despite this decomposition of the support, the function admits no nontrivial decomposition into real rational Herglotz-Nevanlinna functions. More than that, we show that its representing measure is extremal in the full cone of Nevanlinna measures. 

The method by which we obtain this result is generalized as to obtain a whole family of extremal real rational Herglotz-Nevanlinna functions. These are of the form
\begin{equation}\label{eq: additional func in intro}
    -\frac{1}{\varphi(z)+w-\frac{1}{w}},
\end{equation}
where $\varphi$ is a nonconstant one-variable real rational Herglotz-Nevanlinna function. The motivating example \eqref{eq: func in intro} is of this form as well, with $\varphi(z)=z-1/z$.

The last part of the paper connects these results to the conjecture of Knese \cite{KneseConjecturePaper}. In particular, after showing that the example \eqref{eq: func in intro} makes  a counterexample to the irreducibility assertion of the conjecture, we show that the same phenomenon occurs within the extremal family constructed above precisely when the denominator is reducible. We therefore give several characterizations of reducibility of the denominator within this family; for example, the denominator is reducible if and only if the two parametrizing functions of its real zero set are the boundary values of $-U$ and $1/U$ where $U$ is a nonconstant one-variable real rational Herglotz-Nevanlinna function. We end the paper with the construction of an infinite family of counterexamples to Knese's conjecture  by choosing $\varphi = \mathcal{J}^{\circ k}$, $k\geq 1$, in \eqref{eq: additional func in intro} where $\mathcal{J}(w):=w-1/w$.

The paper is organized as follows. In Section~\ref{section: prelims} we recall the Herglotz-Nevanlinna integral representation and related preliminaries. We then give the general operator-theoretic criterion for extremality, establish the basic properties of real rational Herglotz-Nevanlinna functions needed later, and derive the structure formula for their representing measures in two variables. In Section~\ref{section: extremal_theory} we study the motivating example \eqref{eq: func in intro} and use it to develop the family of extremal real rational \HNf s. In Section~\ref{section: conjecture} we pass to the bidisc, obtain a counterexample to the irreducibility assertion of Knese's conjecture, characterize reducibility within the extremal family, and produce an infinite family of counterexamples.

\section{\bfseries Preliminaries,  a criterion for extremality, and the structure formula}\label{section: prelims}

In this section we collect some well known results, such as the Herglotz-Nevanlinna integral representation. Moreover, we recall the  notion of extremality and  we give a new criterion for a Nevanlinna measure on $\mathbb R^2$ to be extremal. We also present the structure formula for the representing measures of real rational \HNf s in two variables.

\subsection{Preliminaries}

We denote by $\mathbb H$ the complex upper  half-plane. By definition a function $h:\mathbb H^d\to\mathbb C$ is a \emph{Herglotz-Nevanlinna function} if it is holomorphic and satisfies $\im h(z)\geq0$ for all $z\in\upperd$. Alternatively, these functions can be characterized via an integral representation. Theorem 4.1 in \cite{Luger_2019} states that a function $h:\uppern\to\C$ is a \HNf{} if and only if it can be written as 
\begin{equation}\label{eq: integral representation}
    h(z)=a+\sum_{j=1}^d b_jz_j+\frac{1}{\pi^d}\int_{\Rn}K_n(z,t)\,d\mu(t),
\end{equation}
where $a\in\R$, $b\in[0,\infty)^d$, $K_d$ is the kernel function
\begin{equation}\label{eq: kernel}
    (z,t)\mapsto\frac{1}{(2i)^{d-1}}\prod_{j=1}^d\left(\frac{1}{t_j-z_j}-\frac{1}{t_j+i}\right)-i\prod_{j=1}^d\frac{1}{t_j^2+1}, \quad (z,t)\in\uppern\times\Rn,
\end{equation}
and $\mu$ is a \emph{Nevanlinna measure}. A Nevanlinna measure on $\Rn$ is a positive Borel measure satisfying the growth condition
\begin{equation}\label{eq: growth condition Nevanlinna measure}
    \int_{\Rn}\prod_{j=1}^d\frac{1}{1+t_j^2}\,d\mu(t)<\infty,
\end{equation}
as well as the \emph{Nevanlinna condition}:
\begin{equation*}
    \int_{\Rn} \frac{1}{(t_{j_1}-z_{j_1})^2(t_{j_2}-\overline{z_{j_2}})^2} \prod_{\substack{j=1\\j\neq j_1,j_2}}^d\frac{2\im z_j}{\abs{t_j-z_j}^2}\,d\mu(t)=0
\end{equation*}
for all $z\in\uppern$ and for every pair of indices $j_1,j_2\in\mathbb{N}$ with $1\leq j_1<j_2\leq d$. In two variables, the Nevanlinna condition reduces to
\begin{equation}\label{eq: nevanlinna condition two variables}
    \int_{\Rtwo}\frac{1}{(u-z)^2(v-\overline{w})^2}\,d\mu(u,v) = 0, \quad  (z,w)\in\uppertwo.
\end{equation}

The representing \emph{data} $(a,b,\mu)$ from \eqref{eq: integral representation} are uniquely determined by $h$.

We recall from \cite[Corollary 4.6]{Luger_2019} also some important consequences of the integral representation.
\begin{corollary}\label{cor: main cor}
    Let $h$ be a \HNf{} of $d$ variables and let $(a,b,\mu)$ be its representing data. Then the following statements hold.
    \begin{enumerate}[label=(\roman*)]
        \item Let $\tau\in\R$. Then, for each $j\in\{1,\dots,d\},$ there is a nonnegative number $c_{z_j}(\tau)$ such that
        \begin{equation*}
            \lim_{z_j\tohat \tau} (\tau-z_j)h(z) =:  c_{z_j}(\tau).
        \end{equation*}
        In particular, the limit exists and is independent of the remaining coordinates. Here the notation $z_j\tohat\tau$  denotes the nontangential limit $z_j\to\tau$ in any \emph{Stolz domain} $\{z\in\upper:\,\theta\leq\arg (z-\tau)\leq \pi-\theta\}$, where $0<\theta\leq\frac{\pi}{2}$.\\ 
        \item The imaginary part of $h$ satisfies
        \begin{equation*}
            \im h(z)=\sum_{j=1}^d b_j\im z_j+\frac{1}{\pi^d}\int_{\Rn}\mathcal{P}_d(z,t)\,d\mu(t), \quad z\in\uppern.
        \end{equation*}
    \end{enumerate}
\end{corollary}

Here the function $\Pn$ is the Poisson kernel for the poly-upper half-plane. It is the product of $d$ copies of the one-dimensional Poisson kernel, and hence
\begin{equation*}
    \Pn(z,t)= \prod_{j=1}^d\frac{\im z_j}{|t_j-z_j|^2}, \quad (z,t)\in\uppern\times\Rn.
\end{equation*}

We finish by recalling  a basic definition for the next section. 

\begin{definition}
    A Nevanlinna measure $\mu$ is \emph{extremal} if it cannot be written in a non-trivial way as a sum of two Nevanlinna measures, that is, if  $\nu_1$ and $\nu_2$ are Nevanlinna measures such that $\nu_1+\nu_2=\mu$, then $\nu_1$ and $\nu_2$ are constant multiples of $\mu$.
\end{definition}
Equivalently, a Nevanlinna measure is extremal if and only if  every decomposition $h=h_1+h_2$ of the corresponding \HNf{} $h$ with zero polynomial part  is of the form $ h_1=c+dh$ and $h_2=-c+(1-d)h$ for some $c\in\R$ and $0\leq d\leq 1$. In this case we also call the corresponding Herglotz-Nevanlinna function \emph{extremal}.

\subsection{An extremality criterion}\label{section: extremality criterion}

In this section we consider the case of two variables only. Let us first view the Nevanlinna condition in a   different way. To this end let $\mu$ be a Nevanlinna measure, which then by assumption satisfies the growth condition \eqref{eq: growth condition Nevanlinna measure}. Hence for every essentially bounded $\mu$-measurable function $f:\Rtwo\to\R$ the integral 
\begin{equation*}
    (\mathcal{T}_\mu f)(z,w):=  \int_{\Rtwo}\frac{f(u,v)}{(u-z)^2(v-\overline{w})^2}\,d\mu(u,v),  \quad (z,w)\in\uppertwo
\end{equation*}
is well-defined and $\mathcal{T}_\mu f$ is a holomorphic function on $\mathbb H^2$. Hence we can introduce the operator
\begin{equation*}
   \mathcal{T}_\mu: L^\infty(\Rtwo,\mathbb R;\mu)\to \mathcal{ O}(\mathbb H^2).
\end{equation*}
Using this notation the Nevanlinna condition \eqref{eq: nevanlinna condition two variables} reads as 
\begin{equation*}
    \mathcal{T}_\mu 1 = 0,
\end{equation*}
i.e., the kernel of $\mathcal{T}_\mu$ contains the constant function. The following theorem rephrases extremality in terms of $\mathcal{T}_\mu$. Let $\mathbb{R}\mathbf{1}$ denote the set of constant real-valued functions.

\begin{theorem}\label{thm: extremal condition}
    Let $\mu$ be a Nevanlinna measure. Then $\mu$ is extremal if and only if $\ker\mathcal{T}_\mu =\mathbb{R}\mathbf{1}$.
\end{theorem}

\begin{proof}
    Suppose $\ker\mathcal{T}_\mu =\mathbb{R}\mathbf{1}$. Write $\mu$ as the sum of two Nevanlinna measures, and let $\nu$ denote any of the summands. Then obviously $\nu$ has to be absolutely continuous with respect to $\mu$ and satisfy $\nu\leq\mu$. By the Radon-Nikodym theorem, there is some $\mu$-measurable $f$ such that $d\nu=fd\mu$. In this case we have $0\leq f\leq 1$ $\mu$-a.e, and hence $f\in L^\infty(\Rtwo,\mathbb R;\mu)$. The Nevanlinna condition implies that $\mathcal{T}_\mu f = \mathcal{T}_\nu 1= 0$. Since $\ker\mathcal{T}_\mu =\mathbb{R}1$, it follows that $f$ is constant. That is, $\nu$ is of the form $c\mu$ for some constant $0\leq c \leq 1$, which proves extremality of $\mu$. 
   
     For the other direction, we show the contrapositive, namely that if $\ker\mathcal{T}_\mu \supsetneq \mathbb{R}1$ then $\mu$ is not extremal. Therefore suppose there is some nonconstant $f\in L^\infty(\Rtwo,\R;\mu)$ such that $\mathcal{T}_\mu f=0$. Since $f$ is essentially bounded, we can choose $0<\varepsilon\leq\|f\|_\infty^{-1}$ such that $1\pm \varepsilon f \geq 0$. In this way, we obtain two positive Borel measures $\nu_\pm $ according to $d\nu_\pm:= (1\pm \varepsilon f) d\mu$. The growth condition is satisfied because $1\pm\varepsilon f$ is bounded, and the Nevanlinna condition is also satisfied:
    \begin{equation*}
        \mathcal{T}_{\pm\nu} 1 = \mathcal{T}_\mu(1\pm\varepsilon f) = \mathcal{T}_\mu 1\pm \varepsilon \mathcal{T}_\mu(f) = 0.
    \end{equation*}
    Hence $\nu_\pm$ are Nevanlinna measures, and they sum to $2\mu$. Since neither measure is a scalar multiple of $\mu$, it follows that $\mu$ is not extremal.
\end{proof}

Note that Theorem \ref{thm: extremal condition} can in a  straightforward way also be generalized to the case of more than $2$ variables, however, then several operators have to be considered and we refrain from this here.

\subsection{Real rational Herglotz-Nevanlinna functions}

A rational function in $d$ variables is called \emph{real rational} if it attains only real values on $\Rn$, whenever the denominator does not vanish. In the rest of this  paper we will consider Herglotz-Nevanlinna functions which are real rational.

\subsubsection{In one variable}\label{section: onevar_prelims}
Real rational Herglotz-Nevanlinna functions in one variable are very well studied. We only want to recall  briefly some facts that will be important for us.

Let $h=\frac{p}{q}$ be a one-variable real rational Herglotz-Nevanlinna function with representing data $(a,b,\mu)$ such that $\deg q = n$. It only has simple poles and these lie on the real line,  and $\mu$ is discrete with finitely many point masses at these poles.  Specifically, 
\begin{equation}\label{eq: one variable real-rational func}
    h(z) = a + bz + \sum_{j=1}^{n}c_j\bigg(\frac{1}{\tau_j-z}-\frac{\tau_j}{\tau_j^2+1}\bigg),
\end{equation}
where each $c_j$ is a positive constant, the numbers $\tau_j$ are the zeros of the denominator $q$, and the measure $\mu = \pi\sum_{j=1}^n c_j\delta_{\tau_j}$. Notice that these $c_j$ agree with those in the one-variable version of Corollary~\ref{cor: main cor}(i). This yields
\begin{equation*}
    c_j = \lim\limits_{z\hat \to \tau_j}\frac{(\tau_j-z)p(z)}{q(z)} =-\frac{p(\tau_j)}{q'(\tau_j)}>0, \quad j=1,\dots,n.
\end{equation*}

\subsubsection{In several variables}

Let $h=p/q$ be a real rational Herglotz-Nevanlinna function and assume  that the numerator and denominator are coprime. We start by recalling some standard notations. The \emph{polydegree} of a polynomial  of $d$ variables is the $d$-tuple $(n_1,\dots \,n_d)$, where  $n_j$ is the degree of  the polynomial in the variable $z_j$. Similarly,  if $h=p/q$ depends on $d$ variables then it has \emph{polydegree}  $(n_1,\dots \,n_d)$ if  $n_j$  is the maximum of the degrees of $p$ and $q$ in the variable $z_j$. In the present paper, we are primarily interested in polynomials and rational functions of two variables, and hence the preferred terminology is \emph{bidegree}. The coordinates we use on $\uppertwo$ are $(z,w)$, as opposed to  $(z_1,\dots,z_d)\in\uppern$ when $d$ is unspecified. When dependence on all specified variables is needed to avoid trivial complications, we state this explicitly, either directly or by requiring the corresponding entries of the polydegree to be positive.

If the polynomial part in  the representation \eqref{eq: integral representation} of $h$ is constant, then \cite[Corollary 4.6(iv)]{Luger_2019} implies that the degree of $q$ is greater than or equal to that of $p$ in each variable. Consequently, the polydegree of $h$ agrees with the polydegree of $q$. The converse also holds; if the polydegree of  $h$ agrees with that of $q$, then its polynomial part is constant.

The following basic fact about real rational Herglotz-Nevanlinna functions will be important  for us and thus we give a proof for completeness. Here and throughout the paper, a property that holds almost everywhere (a.e) does so with respect to the Lebesgue measure, unless otherwise stated. 
We denote the Lebesgue measure by $\lambda$.

\begin{lemma}
    A rational \HNf{} $h=p/q$  is real rational precisely when the coefficients of $p$ and $q$ are real, up to a common nonzero complex constant factor.
\end{lemma} 
\begin{proof}
    Suppose $h$ is real rational and define $p^*(\cdot):=\overline{p(\overline{\cdot})}$ and $q^*(\cdot):=\overline{q(\overline{\cdot})}$. On $\Rn$ we have $p/q=p^*/q^*$ a.e, and therefore
    \begin{equation}\label{eq: real coefficients}
        pq^*=qp^*
    \end{equation}
    on $\Rn$. Since $p$ and $q$ are coprime, it follows that $p\,|\,p^*$ and $q\,|\,q^*$, and hence $p^*=\alpha_1 p$ and $q^*=\alpha_2 q$ for some $\alpha_1,\alpha_2\in\C$. Substituting this back into \eqref{eq: real coefficients} yields $\alpha_1=\alpha_2=:\alpha$. For any two (nonzero) coefficients $c,d$ in $p$ or $q$, we thus have $\overline{c}=\alpha c$ and $\overline{d}=\alpha d$, which yields $\overline{c/d}=c/d$. In particular, $c/d$ is real. It follows that the coefficients of $p$ and $q$ are real up to a common nonzero complex constant factor. Conversely, it is clear that $h$ has real boundary values a.e when the coefficients are real, so the proof concludes.
\end{proof}

Factors of the form  $\tau-z$ in the denominator play  a special role. For later use we recall the following result from \cite[Corollary 3.5]{Luger_2020_geometry_nevanlinna_measure}. 
\begin{lemma}\label{lemma: decomposition}
    Let $h$ be a real rational Herglotz-Nevanlinna function of two variables. Suppose its denominator is divisible by $\tau-z$ for some $\tau\in\R$. Then $h$ can be decomposed as
    \begin{equation}\label{eq: decomposition}
        h(z,w)=\frac{c_z(\tau)}{\tau-z} + \Tilde{h}(z,w),
    \end{equation}
    where $c_z(\tau)>0$ is from Corollary~\ref{cor: main cor}(i) and $\Tilde{h}$ is a real rational Herglotz-Nevanlinna function whose denominator is coprime with $\tau-z$.
\end{lemma}

In other words if the denominator contains a linear factor depending on one variable only, then a kind of partial fraction decomposition into Herglotz-Nevanlinna functions is possible. As we will see in Subsection \ref{section: extremal_example},  this is a rather special situation.

\begin{remark}
    If the denominator of $h$ in Lemma~\ref{lemma: decomposition} contains several distinct factors of the form $\tau-z$, then the decomposition method may be iterated until the denominator of the remaining \HNf{} is coprime with all of them.
\end{remark}

In the following, three important properties of real rational Herglotz-Nevanlinna functions are established. By Lemma~\ref{lemma: decomposition} we can separate the denominator from factors of the form $\tau-z$, and hence we only consider denominators that are coprime with all such factors. However, this separation may leave the remainder $\tilde{h}$ with a nonzero polynomial part, which is why we formulate the following under the assumption that the polynomial part is constant, rather than zero. 
\begin{lemma}\label{lemma: properties of RR functions}
    Let $h=p/q$ be a  real rational Herglotz-Nevanlinna function  of two variables. Suppose its polynomial part  is constant and that $q$ does not contain a factor of the form $\tau-z$ for any $\tau\in\R$. Then the following statements hold.
    \begin{enumerate}[label=(\roman*)]
        \item For $\tau\in\R$, the sliced function $w\mapsto h(\tau,w)$ is a real rational Herglotz-Nevanlinna function of one variable.
        \item The denominator $q$ does not contain factors which  depend on $z$ only.
        \item Let $n$ be the degree of $q$ in the variable $w$. Then there is a finite set $S\subset\R$ such that for $\tau\in\R\setminus S$ the polynomial $w\mapsto q(\tau,w)$ has $n$ distinct real zeros. Moreover, for such $\tau$, the denominator of the sliced function $h(\tau,\cdot)$ is precisely $q(\tau,\cdot)$.
    \end{enumerate} 
\end{lemma}
\begin{proof}
    (i) The slice $h(\tau,\cdot)$ is well-defined because $q(\tau,\cdot) \not\equiv 0$, and it is rational with real coefficients. Since $q(\tau+i\varepsilon,\cdot)\to q(\tau,\cdot)$ uniformly on compact subsets of $\upper$ as $\varepsilon\to 0^+$, and $q$ is nonzero on $\uppertwo$,  Hurwitz's theorem implies that $q(\tau,\cdot)$ is nonzero on $\upper$. In particular, $h(\tau,\cdot)$ is holomorphic on $\upper$. Nonnegativity of its imaginary part is carried over from that of $h$, by continuity.

    (ii) Assume that $r\in\C[z]$ is a nonconstant factor of $q$. Then the  zeros of $r$ have to be real, because otherwise either $h$ has a singularity in $\mathbb H^2$, or $p$ and $q$ are  not coprime.  This, however, contradicts the  assumption on real linear factors in $z$.
    
    (iii) Let $\tau\in\R$ and consider the specialization $q(\tau,\cdot)$, which in (i) was shown to be nonzero on $\upper$. Since the coefficients are real this implies that the zeros of $q(\tau,\cdot)$ are all real. Let $S_1\subset \R$ denote the set of points $\tau$ for which the leading coefficient of $q$, viewed as a polynomial in $w$, vanishes. Then for  $\tau\notin S_1$,  the polynomial $q(\tau,\cdot)$ has $n$ zeros counted with multiplicity. Next define $S_2\subset \R$  the set of points $\tau$ for which $p(\tau,\cdot)$ and $q(\tau,\cdot)$ have factors in common.  Then for $\tau\notin S_2$ the polynomial $q(\tau,\cdot)$ is the denominator of the reduced function $h(\tau,\cdot)$. As a real rational Herglotz-Nevanlinna function, $h(\tau,\cdot)$  has simple poles only, and hence the zeros of $q(\tau,\cdot)$ are simple for such $\tau$. Consequently, $q(\tau,\cdot)$ has $n$ distinct real zeros whenever $\tau\notin S_1\cup S_2=:S$. What is left to show is that the exceptional set $S$ is finite. The set $S_1$ is clearly finite. As for $S_2$, note that for every $\tau\in S_2$, there is some $w_\tau\in\R$ such that $p(\tau,w_\tau)=q(\tau,w_\tau)=0$, meaning that $(\tau,w_\tau)$ is a common zero of $p$ and $q$. But since $p$ and $q$ by assumption are coprime, they have finitely many zeros in common (see e.g. \cite{Elliptic_curves}), hence $S_2$ is finite as well.
    
    In the proof we have  also seen  that for $\tau\notin S$,  the denominator of  $h(\tau,\cdot)$ is $q(\tau,\cdot)$. 
\end{proof}

\subsection{The structure  formula in two variables}

In this section we   give an explicit formula for the representing measure  of a real rational \HNf{} in two variables. It is the Herglotz-Nevanlinna analogue of \cite[Theorems 3.3, 3.8]{clark_measures_two} for Clark measures associated with rational inner functions; however, we derive it within our framework. The proof (including the required lemmas) is rather technical, and hence we move the details to the appendix. However, it is direct, which avoids certain complications compared to \cite{clark_measures_two}. This is particularly apparent in the treatment of vertical components, which in the bidisc leads to the more involved exceptional-case   handled in \cite[Theorem 3.8]{clark_measures_two}. If $\phi=\tilde{p}/p$ is a rational inner function and $\alpha$ is an exceptional parameter value such that $\phi(\zeta,  \cdot)\equiv \alpha$ for some fixed $\zeta\in\mathbb{T}$, then the $\alpha$-level set contains the vertical line $z=\zeta$, and the support of the corresponding Clark measure contains the same line. In the Herglotz-Nevanlinna setting, the analogous vertical components arise from linear factors of the denominator, which can be separated directly by Lemma~\ref{lemma: decomposition}.

Ahead of the structure formula we give the following lemma. It is an analogue of \cite[Lemma 2.5]{clark_measures_two}, and provides a detailed description of the real zero set of the denominator of a real rational Herglotz-Nevanlinna function. As described in the paragraph leading up to Lemma~\ref{lemma: properties of RR functions}, we can separate the denominator from factors of the form $\tau-z$, and hence we only consider denominators that are coprime with all such factors. This is also why we formulate the following under the assumption that the polynomial part is constant, rather than zero.

\begin{lemma}\label{lemma: parametrization of zero set}
    Let $h=p/q$ be a real rational Herglotz-Nevanlinna function of bidegree $(m,n)$, $n\geq 1$, with constant polynomial part such that $q$ and $\tau-z$ are coprime for all $\tau\in\R$. Let $S\subset\R$ be the finite set from Lemma~\ref{lemma: properties of RR functions}(iii). Then there exist real-analytic functions $g_j: \R\setminus S\to\R$, $j=1,\dots,n$, such that $\Z(q)\cap\mathbb R^2$, the real zero set of $q$, is given by the closure of the union of their graphs:
    \begin{equation*}
         \Z(q)\cap\mathbb R^2 = \overline{\bigcup\limits_{j=1}^{n}\left\{(t,g_j(t)): t\in\R\setminus S\right\}}.
    \end{equation*}
\end{lemma}
\begin{proof}
    See Appendix~\ref{app: A}.
\end{proof}

This parametrization of the real zero set leads to the   structure formula for the measure.

\begin{theorem}\label{thm: structure formula representing measures}
    Let $h=p/q$ be a bidegree $(m,n)$, $n\geq 1$, real rational Herglotz-Nevanlinna function and assume the data of $h$ are $(0,0,\mu)$.
    Decompose $h$ according to Lemma~\ref{lemma: decomposition} as
    \begin{equation*}
         h(z,w)=\Tilde{h}(z,w)+\sum_{k=1}^l  \frac{c_z(\tau_k)}{\tau_k-z}
    \end{equation*}
    where $l\geq 0$ is such that the denominator of $\tilde{h}$ is coprime with $\tau-z$ for all $\tau\in\R$. Write $\tilde{h}=\tilde{p}/\tilde{q}$. Let $g_1,\dots,g_n:\R\setminus S\to\R$ be the parametrizing functions of $\Z(\tilde{q})\cap\Rtwo$ from Lemma~\ref{lemma: parametrization of zero set} for $\Tilde h$ and denote by $\gamma_j$ the corresponding graphs in $\mathbb R^2$. For any $g_j$, define the \emph{weight} of $\mu$ associated with $g_j$ as
    \begin{equation}\label{eq: definition weight function}
        \alpha_j(t) := \bigg|\frac{\tilde{p}(t,g_j(t))}{\frac{\partial \tilde{q}}{\partial w}(t,g_j(t))}\bigg|.
    \end{equation}
    Then
    \begin{equation}\label{eq: the structure formula}
        \mu = \pi \sum_{j=1}^n (\gamma_j)_*(\alpha_j\lambda|_{\R\setminus S}) +  \pi\sum_{k=1}^lc_{z}(\tau_k)(\delta_{\tau_k}\times\lambda|_{\R\setminus S}).
    \end{equation}
\end{theorem}
\begin{proof}
    See Appendix~\ref{app: B}.
\end{proof}

\begin{corollary}
    Let $h$ be as in Theorem~\ref{thm: structure formula representing measures} and suppose $\varphi:\Rtwo\to\R$ is Borel measurable and $\mu$-integrable. Then $\mu$ satisfies
\begin{equation*}
    \int_{\Rtwo} \varphi(t)\,d\mu(t)  = \pi \sum_{j=1}^n \int_{\R\setminus S} \varphi(t,g_j(t))\alpha_j(t) \,dt  
     + \pi\sum_{k=1}^lc_{z}(\tau_k)\int_{\R} \varphi(\tau_k,t) \,dt.
\end{equation*}
\end{corollary}

Note that the above applies to any $\varphi\in C^1(\Rtwo)$ satisfying
\begin{equation*}
    \abs{\varphi(x)}\lesssim (1+x_1^2)^{-1}(1+x_2^2)^{-1}, \quad (x_1,x_2)\in\Rtwo,
\end{equation*}
since $\mu$ satisfies the growth condition \eqref{eq: growth condition Nevanlinna measure}.

For later use we note the following property of the weights.

\begin{corollary}\label{cor: weights are stirctly positive}
    Let $h$ be as in Theorem~\ref{thm: structure formula representing measures}. Then the weights of $\mu$ are strictly positive on $\R\setminus S$.
\end{corollary}
\begin{proof}
    Consider a weight $\alpha$ associated with a parametrizing function $g$ and let $\tau\in\R\setminus S$. Then by Lemma~\ref{lemma: properties of RR functions}(iii) we have $\partial_w\tilde{q}(\tau,g(\tau))\neq 0$, and hence $\alpha$ is defined at $\tau$. Lemma~\ref{lemma: properties of RR functions}(iii) further implies that $\tilde{p}(\tau,\cdot)$ and $\tilde{q}(\tau,\cdot)$ are coprime. Since these polynomials are in $\R[w]$ and $\tilde{q}(\tau,g(\tau))=0$, it follows that $\tilde{p}(\tau,g(\tau))$ is nonzero. The weight, which is nonnegative by definition, is therefore also nonzero at $\tau$, and the desired conclusion follows.
\end{proof}

Theorem~\ref{thm: structure formula representing measures} allows us to identify the support of the representing measure of a two-variable real rational Herglotz-Nevanlinna function with the real zero set of its denominator. The Clark measure analogue has already been laid out in \cite{clark_measures_two}.
\begin{corollary}\label{cor: support = zero set}
    Let $h=p/q$ be a real rational Herglotz-Nevanlinna function of two variables and assume the data of $h$ are $(0,0,\mu)$. Then $\supp \mu = \Z(q)\cap\Rtwo$. 
\end{corollary}
\begin{proof}
    In the notation of Theorem~\ref{thm: structure formula representing measures}, we have
     \begin{equation*}
        Z(q)\cap\Rtwo = \big(\bigcup\limits_{k=1}^l\{\tau_k\}\times\R\big)\cup \big(\Z(\tilde{q})\cap\Rtwo\big).
    \end{equation*}
    By Theorem~\ref{thm: structure formula representing measures}, the support of $\mu$ is concentrated on the graphs of the parametrizing functions $g_1,\dots,g_n$ together with the vertical lines $\{\tau_k\}\times\R$, and hence $\supp\mu\subset\Z(q)\cap\Rtwo$. 
    
    As for the reverse inclusion, any point in $\Z(q)\cap\Rtwo$ lies either on a vertical line or in the closure of the set
    \begin{equation*}
       \bigcup_{j=1}^n \{(t,g_j(t)):\,t\in \R\setminus S \}=:G.
    \end{equation*}
    If $(\tau,\sigma)$ lies on a vertical line in $\Z(q)\cap\Rtwo$, then $\tau-z$ is a factor of $q$, and hence the constant $c_z(\tau)$ from Corollary~\ref{cor: main cor}(i) is positive. Therefore, by Theorem~\ref{thm: structure formula representing measures}, any ball $U\ni (\tau,\sigma)$ has positive measure:
    \begin{equation*}
        \mu(U) \geq c_z(\tau)(\delta_{\tau}\times\lambda)(U) >0.
    \end{equation*}
    It follows that $\supp\mu$ contains all vertical components of $\Z(q)\cap\Rtwo$. If $(\tau,\sigma)\in G$, then $(\tau,\sigma)=(\tau,g(\tau))=\gamma(\tau)$ for one of the parametrizing functions $g=g_j$. Let $U\ni (\tau,\sigma)$ be any ball. Since $\tau\notin S$ and $S$ is finite, there is a small interval $I\subset \R\setminus S$ containing $\tau$ such that the graph of $g$ over $I$ is contained in $U$. By Corollary~\ref{cor: weights are stirctly positive} the weight $\alpha$ associated with $g$ is strictly positive on $I$, so Theorem~\ref{thm: structure formula representing measures} gives the ball $U$ positive measure:
    \begin{equation*}
        \mu(U) \geq \gamma_* (\alpha \,\lambda) (U) > 0.
    \end{equation*}
    It follows that $G\subset\supp\mu$, and hence $\overline{G}\subset\supp\mu$. Together with the vertical line case, we arrive at $\supp\mu\supset\Z(q)\cap\Rtwo$.
\end{proof}

\section{\bfseries Extremal Nevanlinna measures for real rational HN-functions}\label{section: extremal_theory}

In Lemma \ref{lemma: decomposition} we have seen that a linear factor of the form $\tau-z$ in the denominator leads to an additive decomposition into rational Herglotz-Nevanlinna functions. Also for functions with denominator of total degree $2$ it can be shown that whenever the denominator factorizes then the function is not extreme. However, it turns out that this does not remain  true  for larger degree.

\subsection{An extremal rational HN-function with a reducible denominator}\label{section: extremal_example}

Consider the real rational function 
 \begin{equation*}
    H(z,w):= -\frac1{z-\frac1z+w-\frac1w}, 
\end{equation*}
which is Herglotz-Nevanlinna as it is a composition of Herglotz-Nevanlinna functions. Simplifying the rational expression yields that the denominator factorizes, namely
 \begin{equation}\label{eq: final example RR function}
        H(z,w)= -\frac{zw}{(z+w)(zw-1)}.
\end{equation}
However, an easy calculation shows that $H$ cannot be decomposed (nontrivially) as a sum of two real rational Herglotz-Nevanlinna functions. As such, its representing measure is extremal in the smaller cone of Nevanlinna measures for real rational Herglotz-Nevanlinna functions. More than that, in what follows we will show that it is also extremal in the full cone of Nevanlinna measures. To this end we  first use Theorem~\ref{thm: structure formula representing measures} in order to determine the measure of \eqref{eq: final example RR function} explicitly. Substituting the measure into the  integral operator from Section~\ref{section: extremality criterion} and making a few appropriate observations ultimately leads to the extremality through Theorem~\ref{thm: extremal condition}. This is remarkable since there do exist Herglotz-Nevanlinna functions with the same denominator as the factors in \eqref{eq: final example RR function}. For example, both functions 
    \begin{equation}\label{eq: examples}
         h_1(z,w):=-\frac{1}{z+w} \qquad\text{ and }\qquad  h_2(z,w):=-\frac{z}{zw-1}
    \end{equation}
are extremal. For $h_1$ this is shown explicitly in \cite{nedic2022structurenevanlinnameasures},  whereas $h_2$ can be written as  $-\frac{1}{-\frac1z+w}$, i.e.,  $h_1$ composed with an automorphism of $\uppertwo$, and hence also is extremal.

\subsubsection{The measure}

We first find the measure $\mu$ of $H$ explicitly. It follows from \cite[Corollary 4.6]{Luger_2019} that the representing data of $H$ are $(0,0,\mu)$. Therefore, Theorem~\ref{thm: structure formula representing measures} applies. The parametrizing functions are $g_1(t)=-t$ and $g_2(t)=1/t$. With $\partial_w ((z+w)(zw-1))=z^2+2zw-1$ we determine the weight of $\mu$ associated with each $g_j$ as
\begin{equation}\label{eq: weights for special function}
    \alpha_1(t)= \frac{t^2}{1+t^2} \quad\text{and}\quad \alpha_2(t)= \frac{1}{1+t^2}.
\end{equation}
Consequently,
\begin{equation}\label{eq: measure for final example}
    \mu = \pi (t,-t)_* \bigg( \frac{t^2}{1+t^2}\lambda|_{\R\setminus \{0\}}\bigg) + \pi (t,1/t)_* \bigg( \frac{1}{1+t^2}\lambda|_{\R\setminus \{0\}}\bigg).
\end{equation}
This can be confirmed by substituting \eqref{eq: measure for final example} into the representation \eqref{eq: integral representation}, which recovers $H$.

\subsubsection{Extremality}

We apply Theorem \ref{thm: extremal condition} in order to show the extremality of $\mu$. To this end let 
\begin{equation}\label{eq: nu is Nevanlinna}
    \int_{\R} \frac{1}{(t-z)^2(1+t^2)}\bigg[ \frac{t^2}{(t+\overline{w})^2} f_1(t) + \frac{t^2}{(1-t\overline{w})^2} f_2(t) \bigg] \, dt = 0, \quad  (z,w)\in\uppertwo,
\end{equation}
where $f_1(t):=f(t,-t)$ and $f_2(t):=f(t,1/t)$. Since $f\in L^\infty(\Rtwo,\mathbb R;\mu)$ and the weights of $\mu$ are positive, $f_1$ and $f_2$ are essentially bounded with respect to not only $\mu$ but also $\lambda$. Without loss of generality we assume $0\leq f_1,f_2 \leq 1$ $\lambda$-a.e For convenience, we write as if we are integrating over all of $\R$, knowing that it really is $\R\setminus S$ with $S=\{0\}$. 

Denote the integrand by $F_{z,w}$. For $w=i\varepsilon$, $\varepsilon > 0$, we find that
\begin{equation}\label{eq: integrand is dominated}
    \abs{F_{z,i\varepsilon}(t)}\leq \frac{1+t^2}{|t-z|^2(1+t^2)} \leq \frac{1}{|t-z|^2}, \quad \text{a.e } t\in\R;
\end{equation}
indeed, since $f_1$ and $f_2$ are essentially bounded by $1$, the first of the above inequalities is obtained from
\begin{equation*}
    \abs{\frac{t^2}{(t-i\varepsilon)^2} f_1(t) + \frac{t^2}{(1+it\varepsilon)^2} f_2(t) } \leq \abs{\frac{t^2}{(t-i\varepsilon)^2}} + \abs{\frac{t^2}{(1+it\varepsilon)^2} } = \frac{t^2}{t^2+\varepsilon^2} + \frac{t^2}{1+(t\varepsilon)^2}\leq 1+t^2.
\end{equation*}
Now, since the right hand side of \eqref{eq: integrand is dominated} defines an integrable function in $t\in\R$, Lebesgue's dominated convergence theorem yields
\begin{equation*}
    \lim_{\varepsilon\to 0^+} \int_{\R} F_{z,i\varepsilon}(t)\,dt  =\int_{\R} F_{z,0}(t)\,dt  =\int_{\R} \frac{f_1(t)+t^2f_2(t)}{(t-z)^2(1+t^2)} \,dt = 0, \quad  z\in\upper.
\end{equation*}
Since $u(t):=\frac{f_1(t)+t^2f_2(t)}{(1+t^2)}\in L^\infty(\R,\R;\lambda)$, Lemma~\ref{lemma: H is identically zero} applies, and hence $u$ is constant a.e. It follows that
\begin{equation}\label{eq: equation 1}
    f_1(t)+t^2f_2(t) = C(1+t^2),\quad \text{a.e } t\in\R,
\end{equation}
for some constant $C\in \R$.

We obtain an additional equation by returning back to \eqref{eq: nu is Nevanlinna} and manipulating it appropriately; first note that for $w=iR$, $R>0$, we have
\begin{equation*}
    R^2 \abs{F_{z,iR}}\leq  \frac{R^2}{|t-z|^2(1+t^2)} \bigg( \frac{t^2}{t^2+R^2} + \frac{t^2}{1+(tR)^2} \bigg) \leq \frac{t^2+1}{|t-z|^2(1+t^2)} \leq  \frac{1}{|t-z|^2},
\end{equation*}
and hence Lebesgue's dominated convergence theorem applies and yields
\begin{equation*}
    \lim_{R\to \infty} -R^2 \int_{\R} F_{z,iR}(t)\,dt  =\int_{\R} \lim_{R\to \infty} -R^2 F_{z,iR}(t)\,dt  =\int_{\R} \frac{t^2f_1(t)+f_2(t)}{(t-z)^2(1+t^2)} \,dt = 0, \quad  z\in\upper.
\end{equation*}
By Lemma~\ref{lemma: H is identically zero}, we arrive at
\begin{equation*}
    t^2f_1(t)+f_2(t) = D(1+t^2),\quad \text{a.e } t\in\R,
\end{equation*}
for some constant $D\in \R$. Together with \eqref{eq: equation 1}, solving for $f_1$ and $f_2$ yields
\begin{equation*}
    f_1(t) = \frac{C-Dt^2}{1-t^2} \quad \text{and} \quad  f_2(t) = \frac{D-Ct^2}{1-t^2},
\end{equation*}
for a.e $t\in\R$ such that $t\neq \pm1$. Given that $f_1$ and $f_2$ are essentially bounded by $1$, we conclude that $C=D=:c$ and therefore also $f_1 = f_2 = c$. It follows that $f\equiv c$ with $0\leq c\leq 1$, and hence Theorem \ref{thm: extremal condition} implies that $\mu$ is extremal.

\subsection{A class of extremal rational HN functions of bidegree $(m,2)$}

In the context of real rational Herglotz-Nevanlinna functions, Theorem~\ref{thm: extremal condition} can be refined using Theorem~\ref{thm: structure formula representing measures}. If the representing measure contains a nonzero vertical term in addition to a nonvertical part, then Lemma~\ref{lemma: decomposition} yields a nontrivial decomposition, and hence the measure is not extremal. Thus, when studying extremality, we may restrict attention to real rational Herglotz-Nevanlinna functions whose representing measures have no vertical part. In this case,  for $f\in L^\infty(\Rtwo,\R;\mu)$, Theorem~\ref{thm: structure formula representing measures} yields
\begin{equation}\label{eq: operator on real-rational function}
    (\mathcal{T}_\mu f)(z,w)=  \pi  \int_{\R\setminus S} \frac{1}{(t-z)^2} \sum_{j=1}^n\frac{f_j(t)\alpha_j(t)}{(g_j(t)-\overline{w})^2}\,dt, \quad (z,w)\in\uppertwo,
\end{equation}
where for each $j$, $f_j(t):=f(t,g_j(t))$. As seen in Section~\ref{section: extremal_example}, $f_j\in L^\infty(\R,  \R;\lambda)$ for each $j$  since $\alpha_j>0$ on $\R\setminus S$   by Corollary~\ref{cor: weights are stirctly positive}. By Theorem~\ref{thm: extremal condition}, $\mu$ is extremal if and only if whenever  the integral \eqref{eq: operator on real-rational function} vanishes on $\uppertwo$, there is some constant $c\in\R$ (depending on $f$) such that $(f_1,\dots,f_n)=(c,\dots,c)$ for $\lambda$-a.e $t\in\R$.

In the following, we obtain a family of extremal measures from a generalization of the method considered in Section~\ref{section: extremal_example}, applied to the specialized operator \eqref{eq: operator on real-rational function}. The associated subclass of bidegree $(m,2)$ real rational Herglotz-Nevanlinna functions is made explicit.

\begin{theorem}\label{thm: extremal measure}
    Suppose $\varphi$ is a nonconstant one-variable real rational Herglotz-Nevanlinna function. Then
    \begin{equation}\label{eq: general form of extremal function}
        h(z,w)= -\frac{ 1}{\varphi(z)+ w-\frac{1}{w}}
    \end{equation}
    is a real rational Herglotz-Nevanlinna function, and its representing measure  is extremal. 
\end{theorem}

\begin{proof}
    Since $\varphi$ and $\mathcal{J}(w):=w-1/w$ are real rational \HNf s, $h$ is also a real rational HN-function.

    To fix notation, in view of  Theorem \ref{thm: structure formula representing measures} the measure $\mu$ of the function $h$ in \eqref{eq: general form of extremal function} is supported on the graphs of two functions, $g_1$ and $g_2$, with weights $\alpha_1$ and $\alpha_2$. From here, the proof is split into three parts. In the first part, we show that $g_j$ and $\alpha_j$ satisfy the  identities
    \begin{equation}\label{eq: boundary property}
            \frac{\alpha_1}{g_1^2} +  \frac{\alpha_2}{g_2^2}  =  \alpha_1+\alpha_2 \equiv  1
    \end{equation}
    and $g_1g_2\equiv -1$. In the second part, we derive an additional property. In the third and final part, we show that these imply that $\mu$ is extremal.

    \textbf{Part 1}
    Consider $h$ on the reduced form $p/q$ with $p,q\in \C[z,w]$. We can write $q(z,w)=A(z)w^2+B(z)w-A(z)$ for some coprime $A,B\in\R[z]$ such that $\varphi=B/A$. Also note that $p(z,w)=-wA(z)$. By Vieta's formulas, the parametrizing functions $g_1$ and $g_2$ thus satisfy
    \begin{equation}\label{eq: vieta's formulas}
        g_1+g_2=-\varphi \quad\text{and}\quad g_1g_2\equiv -1.
    \end{equation}
    Combining these two gives
    \begin{equation}\label{eq: relation phi and g}
        \varphi=\frac{1-g^2}{g}
    \end{equation}
    for any choice $g=g_j$. 

    Consider $h-\re h(i,i)$ and note that its representing data are $(0,0,\mu)$. Thus, $h-\re h(i,i)$ satisfies the assumptions of Theorem~\ref{thm: structure formula representing measures}.
   Moreover, $q$ has no factor of the form $\tau-z$, since $A$ and $B$ are coprime. 
    Thus, the weight of $\mu$ associated with $g$ is
    \begin{equation*}
         \alpha(t) = \bigg|\frac{p(t,g(t))}{\frac{\partial q}{\partial w}(t,g(t))}\bigg| = \bigg|\frac{g(t)A(t)}{2A(t)g(t)+B(t)}\bigg|=\bigg|\frac{g(t)}{2g(t)+\varphi(t)}\bigg|,
    \end{equation*}
    which reduces to
    \begin{equation}\label{eq: weight associated with g}
        \alpha = \frac{g^2}{1+g^2}=\frac{ 1}{\mathcal{J}'(g)}
    \end{equation}
    after substituting \eqref{eq: relation phi and g} into the expression. Combining \eqref{eq: weight associated with g} with $g_1g_2\equiv -1$ yields the desired   identities \eqref{eq: boundary property}.

    \textbf{Part 2}
    By the structure of nonconstant one-variable real rational Herglotz-Nevanlinna functions (see Section~\ref{section: onevar_prelims}), there is a sequence $(t_n)\subset\R$ such that $\varphi(t_n)\to 0$ as $n\to\infty$. In fact, given a $\lambda$-null set in $\R$, we can find such a sequence away from it; this observation is relevant in part $3$, and also allows us to assume $(t_n)\cap(S\cup \Z(A))=\varnothing$. Choose a parametrizing function $g=g_j$, then
    \begin{equation*}
        q(t_n,g(t_n))=A(t_n)(g^2(t_n)+\varphi(t_n)g(t_n)-1) = 0 \quad \forall n\in\mathbb{N},
    \end{equation*}
    and hence $g(t_n)^2+\varphi(t_n)g(t_n)-1=0$ for all $n\in\mathbb{N}$. In particular,
    \begin{equation}\label{eq: additional boundary property}
        g^2(t_n)\to 1 \quad\text{as}\quad n\to\infty,
    \end{equation}
    which constitutes our additional property.
     
    \textbf{Part 3} 
    In this part, we show that  \eqref{eq: boundary property} and \eqref{eq: additional boundary property} imply that $\mu$ is extremal. To this end, let $f\in\ker\mathcal{T}_\mu$ be arbitrary. In accordance with the discussion preceding Theorem~\ref{thm: extremal measure}, we have
    \begin{equation}\label{eq: rewritten nevanlinna condition for part 2}
        (\mathcal{T}_\mu f)(z,w)= \pi \int_{\R\setminus S} \frac{1}{(t-z)^2} \sum_{j=1}^2\frac{f_j(t)\alpha_j(t)}{(g_j(t)-\overline{w})^2}\,dt = 0, \quad (z,w)\in\uppertwo.
    \end{equation}
    The aim is to show that $f$ is constant, i.e $(f_1,f_2)=(c,c)$ for some constant $c$. Extremality then follows from Theorem~\ref{thm: extremal condition}.
    
    Since $f$ is essentially bounded, there is some $M\in\R$ such that $|f_j|\leq M$ for $j=1,2$ and a.e $t\in\R$. Therefore, for $w=i\varepsilon$, $\varepsilon > 0$, we have for a.e $t\in\R$,
    \begin{align*}
        \bigg| \frac{1}{(t-z)^2} \sum_{j=1}^2\frac{f_j(t)\alpha_j(t)}{(g_j(t)-\overline{w})^2} \bigg| & \leq \frac{1}{|t-z|^2}  \sum_{j=1}^2 \bigg| \frac{f_j(t)\alpha_j(t)}{(g_j(t) +i\varepsilon)^2} \bigg| \\
         & \leq \frac{M}{|t-z|^2}  \sum_{j=1}^2 \frac{\alpha_j(t)}{|g_j(t)+i\varepsilon|^2} \\
        & \leq \frac{M}{|t-z|^2}  \sum_{j=1}^2 \frac{\alpha_j(t)}{g_j^2(t)+\varepsilon^2} \leq \frac{M}{|t-z|^2}  \sum_{j=1}^2 \frac{\alpha_j(t)}{g_j^2(t)}.
    \end{align*}
    Similarly, for $w=iR$, $R>0$, we have for a.e $t\in\R$,
     \begin{equation*}
        R^2\bigg| \frac{1}{(t-z)^2} \sum_{j=1}^2\frac{f_j(t)\alpha_j(t)}{(g_j(t)-\overline{w})^2} \bigg| \leq \frac{M}{|t-z|^2}  \sum_{j=1}^2 \frac{R^2\alpha_j(t)}{g_j^2(t)+R^2} \leq \frac{M}{|t-z|^2}  \sum_{j=1}^2 \alpha_j(t).
    \end{equation*}
    By \eqref{eq: boundary property}, the right-hand sides are integrable. Lebesgue's dominated convergence theorem thus yields
    \begin{align*}
        \lim_{\varepsilon\to 0^+}(\mathcal{T}_\mu f)(z,i\varepsilon) & = \pi \int_{\R\setminus S} \frac{1}{(t-z)^2}  \lim_{\varepsilon\to 0^+} \sum_{j=1}^2\frac{f_j(t)\alpha_j(t)}{(g_j(t)+i\varepsilon)^2}\,dt \\
        & = \pi\int_{\R\setminus S} \frac{1}{(t-z)^2} \sum_{j=1}^2\frac{f_j(t)\alpha_j(t)}{g_j^2(t)}\,dt,
    \end{align*}
    and
    \begin{align*}
        \lim_{R\to\infty}-R^2(\mathcal{T}_\mu f)(z,iR) & = \pi \int_{\R\setminus S} \frac{1}{(t-z)^2}  \lim_{R\to \infty} \sum_{j=1}^2\frac{-R^2f_j(t)\alpha_j(t)}{(g_j(t)+iR)^2}\,dt \\
        & =\pi \int_{\R\setminus S} \frac{1}{(t-z)^2} \sum_{j=1}^2f_j(t)\alpha_j(t)\,dt.
    \end{align*}
    In accordance with \eqref{eq: rewritten nevanlinna condition for part 2}, both limits are zero for all $z\in\upper$. By \eqref{eq: boundary property}, both sums are essentially bounded, and hence Lemma~\ref{lemma: H is identically zero} yields constants $c_0,c_\infty\in\R$ such that
    \begin{equation*}
        \sum_{j=1}^2\frac{f_j(t)\alpha_j(t)}{g_j^2(t)}=c_0 \quad\text{and}\quad  \sum_{j=1}^2f_j(t)\alpha_j(t)=c_\infty
    \end{equation*}
    for a.e $t\in\R$. This can be written in matrix form as  $M(f_1\, f_2)^T = (c_0 \, c_\infty)$, where
    \begin{equation*}
        M=
        \begin{bmatrix}
            \frac{\alpha_1}{g_1^2} & \frac{\alpha_2}{g_2^2} \\
            \alpha_1 & \alpha_2
        \end{bmatrix}.
    \end{equation*}
    From the additional  property \eqref{eq: additional boundary property}, we can conclude that $c_0=c_\infty$; indeed, subtracting the two equations yields
    \begin{equation}\label{eq: c0-cinfty}
        \abs{c_0-c_\infty } = \bigg|\alpha_1f_1\bigg(\frac{1}{g_1^2}-1\bigg)+\alpha_2f_2\bigg(\frac{1}{g_2^2}-1\bigg)\bigg|\lesssim \bigg|\frac{1}{g_1^2}-1 \bigg|+\bigg|\frac{1}{g_2^2}-1\bigg|,
    \end{equation}
    where the inequality holds a.e by boundedness a.e of each $\alpha_j$ (a consequence of \eqref{eq: boundary property} together with Corollary~\ref{cor: weights are stirctly positive}) and each $f_j$. Recall from part $2$ that we can choose a sequence $(t_n)\subset\R$ away from where the inequality may fail, such that for $j=1,2$, $g_j^2(t_n)\to 1$ as $n\to\infty$. Evaluating the right hand side of \eqref{eq: c0-cinfty} along this sequence and taking the limit $n\to\infty$ then yields the desired outcome $c_0=c_\infty$.

   Let $c$ be the common value of $c_0$ and $c_\infty$, and note that by  \eqref{eq: boundary property}, $M(1\, 1)^T = (1 \, 1)$. Thus
    \begin{equation}\label{eq: matrix equation}
        M
        \begin{bmatrix}
            f_1 \\
            f_2
        \end{bmatrix}
        =
         \begin{bmatrix}
            c_0 \\
            c_\infty
        \end{bmatrix}
        = 
        \begin{bmatrix}
            c \\
            c
        \end{bmatrix}
        = 
         M
        \begin{bmatrix}
             c \\
             c
        \end{bmatrix}.
    \end{equation}
    Consider 
    \begin{equation*}
        \det M = \alpha_1\alpha_2\bigg(\frac{1}{g_1^2}-\frac{1}{g_2^2}\bigg) = \alpha_1\alpha_2 \bigg(\frac{1}{g_1^2}-g_1^2\bigg) = \frac{\alpha_1\alpha_2}{g_1^2} (1-g_1)(1+g_1)(1+g_1^2).
    \end{equation*}
    Each factor is nonzero a.e; indeed, the weights $\alpha_1$, $\alpha_2$ are strictly positive on $\R\setminus S$, as per Corollary~\ref{cor: weights are stirctly positive}, and since $\varphi$ is nonconstant, \eqref{eq: relation phi and g} implies that the remaining factors are nonzero a.e. Consequently, $M$ is invertible a.e, and hence $(f_1,f_2)=(c,c)$ solves \eqref{eq: matrix equation} for a.e $t\in\R$. Since $f\in\ker\mathcal{T}_\mu$ was arbitrary, the conclusion follows from Theorem~\ref{thm: extremal condition}.
\end{proof}

\begin{remark}\label{rem: graphs never intersect}
    The  identity $g_1g_2\equiv -1$ derived in part $1$ shows that the structure of the function forces the supporting curves of its representing measure to be disjoint; indeed, there is no point $t\in\R\setminus S$ such that $g_1(t)=g_2(t)$, because then $g_1(t)g_2(t)=g_1^2(t)$ which is nonnegative.
\end{remark}

\begin{remark}\label{rem: weight of extremal function in special class}
    From the proof we find that the weight $\alpha$ associated with the parametrizing function $g$ is given by $1/\mathcal{J}'(g)$, where $\mathcal{J}(w):=w-1/w$. Observe that the weight  depends only on  $g$. Thus,  the real zero set of the denominator  completely determines the representing measure of a real rational Herglotz-Nevanlinna function of the form \eqref{eq: general form of extremal function}. Corollary~\ref{cor: support = zero set} identifies this zero set with the support of the measure, and hence the support determines the measure. 
\end{remark}

The preceding proof shows that functions of the form \eqref{eq: general form of extremal function} satisfy a specific set of  identities. The following shows that these identities essentially characterize the family.

\begin{proposition}\label{prop: boundary properties of normal form}
    Let $h=p/q$ be a  real rational Herglotz-Nevanlinna function of bidegree $(m,2)$, $m \geq 1$, with constant polynomial part such that $q$ and $\tau-z$ are coprime for all $\tau\in\R$. Then $h$  agrees with a function of the form \eqref{eq: general form of extremal function} up to addition by a real constant, with some nonconstant real rational HN-function $\varphi$, if and only if it satisfies \eqref{eq: boundary property} and $g_1g_2\equiv -1$.
\end{proposition}
\begin{proof}
    Part $1$ of the proof of Theorem~\ref{thm: extremal measure} shows that $h$ satisfies the  identities if it is of the form \eqref{eq: general form of extremal function}. We show the converse here.

    Assume $h$ satisfies \eqref{eq: boundary property}  and $g_1g_2\equiv -1$. Since the degree of $h$ in $w$ is two, we have $q(z,w)=A(z)w^2+B(z)w+C(z)$ for some $A,B,C\in\R[z]$. By Vieta's formulas, the parametrizing functions $g_1,g_2$ from Lemma~\ref{lemma: parametrization of zero set} satisfy
    \begin{equation*}
        g_1+g_2=-\frac{B}{A} \quad\text{and}\quad g_1g_2=\frac{C}{A}.
    \end{equation*}
    We obtain $C=-A$ from  $g_1g_2\equiv -1$. It follows that $q(z,i\varepsilon)$ converges uniformly on compact subsets to $-A(z)$ as $\varepsilon\to 0^+$, and hence Hurwitz's theorem implies that $A$ is nonzero on $\upper$. This allows us to write
    \begin{equation*}
        q(z,w)=A(z)w(\varphi(z)+\mathcal{J}(w)), \quad (z,w)\in\uppertwo,
    \end{equation*}
    where $\varphi:=\frac{B}{A}$ is holomorphic on $\upper$ and $\mathcal{J}(w):=w-1/w$. Since $q\neq 0$ on $\uppertwo$, we have $\varphi(z)+\mathcal{J}(w)\neq 0$ on $\uppertwo$. This implies that $\im\varphi\geq 0$ on $\upper$, since $\mathcal{J}(w)\in\upper$ whenever $w\in\upper$ and $\mathcal{J}$ is surjective on $\upper$. It follows that $\varphi$ is a one-variable real rational Herglotz-Nevanlinna function. To see that $\varphi$ is nonconstant, assume the contrary for a contradiction: then $B=cA$ for some $c\in\R$, which yields $q(z,w)=A(z)(w^2+cw-1)$. By assumption, $m\geq 1$ and $q$ is coprime with factors of the form $\tau-z$, and hence Lemma~\ref{lemma: properties of RR functions}(ii) forces $A$ to be constant. As such, $B=cA$ is also constant, which contradicts  $m\geq 1$. In conclusion, $\varphi$ is nonconstant. 
    
    From the above we conclude that $q$ has the exact same structure as the denominator of some function $\tilde{h}$ of the form \eqref{eq: general form of extremal function}. In particular, the representing measures of $h$ and $\tilde{h}$ have the same support. Furthermore, a quick calculation shows that \eqref{eq: boundary property} and $g_1g_2\equiv -1$ together imply that the weight $\alpha = \alpha_j$ associated with $h$ satisfies \eqref{eq: weight associated with g}, meaning that the weights of the two measures agree. Denote by $h_k$ and $\tilde{h}_k$ the polynomial parts of $h$ and $\tilde{h}$, respectively. It follows from Theorem~\ref{thm: structure formula representing measures} that $h-h_k$ and $\tilde{h}-\tilde{h}_k$ have the same representing measure, and hence  uniqueness of the representing data yields $h=\tilde{h}$ up to addition by a real constant. 
\end{proof}

\section{\bfseries Counterexamples to Knese's conjecture}\label{section: conjecture}

In this last section we connect our results to a conjecture by Knese \cite[Conjecture 1.7]{KneseConjecturePaper}. Since it is formulated for functions defined in the setting of positive real functions we first need to introduce some notation.

Consider a bidegree $(m,n)$ polynomial $p\in\C[z,w]$. The \emph{reflection} of $p$ is defined as the polynomial
\begin{equation*}
    \tilde{p}(z,w):=z^{m}w^{n}\overline{p(1/\overline{z},1/\overline{w})}.
\end{equation*}
A polynomial $p\in\C[z,w]$ is \emph{scattering stable} if it is nonzero on $\D^2$ and coprime with the reflection of $p$. A bidegree $(m,n)$ scattering stable polynomial $p$ is $\mathbb{T}^2$-\emph{saturated} if it has $2mn$ zeros on $\mathbb{T}^2$ in common with its reflection (counting with intersection multiplicities). Let \emph{RHP} denote the open right half plane in $\C$, and let $\text{\emph{PRP}}_2$ be the set of holomorphic functions $f:\mathbb{D}^2\to \text{\emph{RHP}}$ such that $f(0)=1$. This set  is obviously convex and the conjecture concerns certain extreme elements.

\begin{conjecture}\label{conj: Knese}
     If $p\in\C[z,w]$ is scattering stable and if $(p+\tilde{p})/(p-\tilde{p})$ is an extreme point of $\text{\emph{PRP}}_2$, then $p$ is $\mathbb{T}^2$-saturated and $p-\tilde{p}$ is irreducible.
\end{conjecture}

We will show that the irreducibility aspect of Conjecture~\ref{conj: Knese} is false. We first construct a counterexample from the function $H$ in \eqref{eq: final example RR function}, and then characterize when the family from Theorem~\ref{thm: extremal measure} yields further counterexamples.

\subsection{A first counterexample}\label{section: first_counterexample}

Let $\gamma$ be the Cayley map $z\mapsto i(1+z)/(1-z)$, mapping $\D\to\upper$.  Then with $H$ from \eqref{eq: final example RR function}, define the function
\begin{equation*}
    f(z,w) :=-4iH(\gamma(z), \gamma(w)) = \frac{(1-z^2)(1-w^2)}{1-z^2w^2},
\end{equation*}
which belongs to $\text{\emph{PRP}}_2$. With $p(z,w):=2-z^2-w^2$ and its reflection $\tilde{p}(z,w)=2z^2w^2-z^2-w^2$ the function $f$ can be written as $(p+\tilde{p})/(p-\tilde{p})$. Note that $p$ is nonzero on $\mathbb{D}^2$ since
\begin{equation*}
    \abs{z^2+w^2}\leq \abs{z}^2+\abs{w}^2 < 2
\end{equation*}
whenever $\abs{z},\abs{w}<1$. Moreover, $p$ and $\tilde{p}$ are coprime; indeed, the polynomials
\begin{equation*}
    (z,w)\mapsto 2-z-w \quad\text{and}\quad (z,w)\mapsto 2zw-z-w
\end{equation*}
have a single common zero $(1,1)$, and hence the common zeros of $p$ and $\tilde{p}$ are the finitely many points $(\pm 1,\pm 1)$. As such, $p$ is scattering stable. 

In the following we show that $f$ is an extreme point in $\text{\emph{PRP}}_2$. Write $f=(f_1+f_2)/2$ for some $f_1,f_2\in \text{\emph{PRP}}_2$. For each $j$, the composition $if_j(\gamma^{-1},\gamma^{-1})/4$ defines a \HNf{} $H_j$. It follows that
\begin{equation*}
    \frac{H_1+H_2}{2} = \frac{\frac{i}{4}f_1(\gamma^{-1},\gamma^{-1})+\frac{i}{4}f_2(\gamma^{-1},\gamma^{-1})}{2} = \frac{i}{4}f(\gamma^{-1},\gamma^{-1}) = H,
\end{equation*}
and hence $H=(H_1+H_2)/2$. By the extremality of $h$, this yields $H_1=c+d_1 h$ and $H_2=-c + d_2h$ with $c\in\R$ and $ d_1,d_2\geq 0$. We thus find that $f_1 = -4i H_1 (\gamma,\gamma)= -4ci + d_1f$, and similarly $f_2 = 4ci + d_2f$. From $f_1(0)=f(0)=1$ we obtain $-4ci+d_1=1$, which by comparison of real and imaginary parts yields $c=0$ and $d_1=1$. Similarly, $f_2(0)=1$ yields $d_2=1$. Consequently, $f_1 = f_2 = f$.

In conclusion, $p$ is scattering stable and $f=(p+\tilde{p})/(p-\tilde{p})$ is an extreme point of $\text{\emph{PRP}}_2$. As such, all the assumptions in Conjecture~\ref{conj: Knese} are satisfied. The saturation aspect does hold in this example; indeed, $p$ is $\mathbb{T}^2$-saturated, which follows from saturation of $2-z-w$ (considered in \cite{KneseConjecturePaper}) under the substitution $z\mapsto z^2$ and $w\mapsto w^2$. However, more importantly, $p-\tilde{p}$ is not irreducible:
\begin{equation*}
    p-\tilde{p} = 2-2z^2w^2=2(1-zw)(1+zw).
\end{equation*}
A counterexample to the conjecture has thus been constructed.

\subsection{Characterizations of reducibility and a family of counterexamples}\label{section: reducibility}

In what follows we generalize this example and characterize when the denominator of the extremal function \eqref{eq: general form of extremal function} from Theorem~\ref{thm: extremal measure} is reducible.
\begin{proposition}\label{prop: reducibility}
    Let $h=p/q$ be as in Theorem~\ref{thm: extremal measure}. Then the following are equivalent:
    \begin{enumerate}[label=(\roman*)]
        \item $q$ is reducible over $\R$, or equivalently over $\C$,
        \item $\varphi^2+4$ is a square in $\R(z)$,
        \item There exists a nonconstant one-variable real rational Herglotz-Nevanlinna function $U$ such that $\varphi = \mathcal{J}\circ U$,
        \item There exists a nonconstant one-variable real rational Herglotz-Nevanlinna function $U$ such that the parametrizing functions $g_1,g_2$ from Lemma~\ref{lemma: parametrization of zero set} are given by $\{g_1,g_2\}=\{-U,1/U\}$, up to relabelling.
    \end{enumerate}
\end{proposition}

\begin{proof}
    We show $(i)\Rightarrow (ii) \Rightarrow (iii) \Rightarrow (iv) \Rightarrow (i)$. In the first implication, we will in fact only assume that $q$ is reducible over $\C$. Together with the final implication, which concludes that $q$ is reducible over $\R$, this will also establish the equivalence of reducibility of $q$ over $\R$ and $\C$.

    $\mathbf{(i)\Rightarrow (ii)}$: Assume $q$ is reducible over $\C$. Recall from part one of the proof of Theorem~\ref{thm: extremal measure} that we can write
    \begin{equation*}
        q(z,w)=A(z)w^2+B(z)w-A(z) = A(z)(w^2+\varphi(z)w-1)
    \end{equation*}
    for some $A,B\in\R[z]$. 
    The degree in $w$ of each factor in any nontrivial factorization of $q$ is $1$; otherwise one factor would depend only on $z$, which cannot happen by Lemma~\ref{lemma: properties of RR functions}(ii). 
    
    Thus, $w^2+\varphi(z)w-1$ has to be reducible in $\C(z)[w]$, and hence its discriminant $\varphi^2+4$ is a square in $\C(z)$. Let $R\in\C(z)$ such that $\varphi^2+4=R^2$, then either $R=\overline{R}$ or $R=-\overline{R}$. In the first case, $R\in\R(z)$. The second case cannot happen because then $R=iS$ for some $S\in\R(z)$, which yields a contradiction since then $\varphi^2+4=-S^2>0$ at real values. Consequently, $\varphi^2+4$ is a square in $\R(z)$. 
    
    $\mathbf{(ii)\Rightarrow (iii)}$: Since $\varphi^2+4$ is a square in $\R(z)$, there is some $R\in\R(z)$ such that $\varphi^2+4=R^2$. Rewrite this as $(R-\varphi)(R+\varphi)=4$ and let $U:=(R+\varphi)/2$. Then $1/U=(R-\varphi)/2$, which yields
    \begin{equation*}
        \mathcal{J}(U) = U-\frac{1}{U} = \frac{R+\varphi}{2} - \frac{R-\varphi}{2} = \varphi.
    \end{equation*}
    It remains to show that $U$ is a nonconstant real rational Herglotz-Nevanlinna function: all desired properties but $\im U\geq 0$ on $\upper$ are immediate from $\mathcal{J}(U)=\varphi$, and $\im U\geq 0$ follows from
    \begin{equation*}
        \im \mathcal{J}(U)=\im \bigg(U-\frac{1}{U}\bigg)=\im\bigg( U-\frac{\overline{U}}{|U|^2}\bigg)=\im U\bigg(1+\frac{1}{|U|^2}\bigg)
    \end{equation*}
    together with  $ \im \mathcal{J}(U)=\im \varphi \geq 0$.

    $\mathbf{(iii)\Rightarrow (iv)}$: Since $U$ is real rational and nonconstant, there are  coprime $P,Q\in \R[z]$ such that $U=P/Q$. From $\mathcal{J}(U)=\varphi$ we obtain $\varphi=(P^2-Q^2)/PQ$, which is in reduced form. Recall from part one of Theorem~\ref{thm: extremal measure} that $\varphi=B/A$, where $A,B\in\R[z]$ are coprime such that  
    \begin{equation*}
        q(z,w)=A(z)w^2+B(z)w-A(z) = A(z)(w^2+\varphi(z)w-1).
    \end{equation*}
    It follows that $A$ is a nonzero constant multiple of $PQ$, and hence $A$ vanishes precisely where $U$ either vanishes or has a pole. Since $A$ is the leading coefficient of $q$ when $q$ is written as a polynomial in $w$, its zero set is contained in the exceptional set $S$ (recall $S_1\subset S$ from Lemma~\ref{lemma: properties of RR functions}(iii)). Therefore, for $t\in\R\setminus S$ and a parametrizing function $g=g_j$, the desired conclusion is obtained from
    \begin{align*}
        q(t,g(t)) & =A(t)(g^2(t)+\varphi(t)g(t)-1)\\
        & = A(t)\bigg( g^2(t)+\bigg(U(t)-\frac{1}{U(t)}\bigg)g(t)-1\bigg) = A(t) (g(t)+U(t))\bigg(g(t)-\frac{1}{U(t)}\bigg)=0.
    \end{align*}

    $\mathbf{(iv)\Rightarrow (i)}$: This implication is routed via (iii). Start by assuming $\{g_1,g_2\}=\{-U,1/U\}$. Then, by Vieta's formula $g_1+g_2=-\varphi$, we find that
    \begin{equation*}
        \varphi = -\bigg(-U+\frac{1}{U}\bigg)=\mathcal{J}(U).
    \end{equation*}
    With this we have
        \begin{equation*}
        h(z,w)=-\frac{1}{\mathcal{J}(U)+\mathcal{J}(w)} = -\frac{U(z)w}{(U(z)+w)(U(z)w-1)}.
    \end{equation*}
    Since $U$ is real rational and nonconstant, there are coprime $P,Q\in \R[z]$ such that $U=P/Q$. Clearing denominators in the above through multiplication by $Q^2$ results in
    \begin{equation*}
         h(z,w)=-\frac{P(z)Q(z)w}{(P(z)+wQ(z))(P(z)w-Q(z))}.
    \end{equation*}
    At the zeros $w=-\frac{P}{Q}$ and $w=\frac{Q}{P}$ of the denominator, the numerator evaluates to $P^2\notequiv 0$ and $-Q^2\notequiv 0$, respectively. Therefore, none of the factors $P+wQ$ and $Pw-Q$ divide the numerator. These factors are irreducible because any factorization of either one would have a factor independent of $w$, which then would divide both $P$ and $Q$, contradicting coprimality of $P$ and $Q$. As such, the above is in reduced form. It follows that $q$ is a nonzero constant multiple of $(P+wQ)(Pw-Q)$, which clearly is reducible over $\R$.
\end{proof}

\begin{remark}\label{rem: general extremal measure}
    If $h$ is as in Theorem~\ref{thm: extremal measure} such that any of the conditions in Proposition~\ref{prop: reducibility} hold, then its representing measure $\mu$ is extremal and given by
    \begin{equation*}
        \mu = \pi (t,-U(t))_* \bigg( \frac{U^2}{1+U^2}\lambda|_{\R\setminus S}\bigg) + \pi  (t,1/U(t))_* \bigg( \frac{1}{1+U^2}\lambda|_{\R\setminus S}\bigg),
    \end{equation*}
    where $U$ is a nonconstant one-variable real rational Herglotz-Nevanlinna function related to the parametrizing functions $g_1,g_2$ of the denominator of $h$ through $g_1=-U$ and $g_2=1/U$ on $\R\setminus S$. As described in Remark~\ref{rem: graphs never intersect}, the two parametrizing functions never intersect, and hence the measure  is a sum of two positive Borel measures  with disjoint supporting curves. They both satisfy the growth condition \eqref{eq: growth condition Nevanlinna measure}, but they do not individually satisfy the Nevanlinna condition \eqref{eq: nevanlinna condition two variables}. Interestingly, however, is that their disjoint supporting curves do support Nevanlinna measures, e.g. $\pi (t,-U(t))_*\lambda$ and $\pi (t,1/U(t))_*\lambda$. This observation matches with what we already know about the example \eqref{eq: final example RR function}.
\end{remark}

\begin{corollary}
    Let $h$ be as in Theorem~\ref{thm: extremal measure}  with zero polynomial part. Then $h$ produces a counterexample to the irreducibility aspect of Conjecture~\ref{conj: Knese} precisely when any of the conditions in Proposition~\ref{prop: reducibility} hold. 
\end{corollary}
\begin{proof}
    Define   
    \begin{equation*}
        f(z,w) :=-\frac{i}{\eta}h(\gamma(z),\gamma(w)), \quad \eta:=\im h(i,i)>0,
    \end{equation*}
    where $\gamma:\D\to\upper$ is the Cayley transform $z\mapsto i(1+z)/(1-z)$. This is the bidisc equivalent of $h$; $f$ is holomorphic on $\D^2$, rational, and satisfies $\re f > 0$ on $\D^2$. It has also been normalized so that
    \begin{equation*}
        f(0) = -\frac{i}{\eta}h(i,i)=  -\frac{i}{\eta}(\re h(i,i)+i\im h(i,i)) = -\frac{i}{\eta} \re h(i,i) +1 =1,
    \end{equation*}
    where $\re h(i,i) = 0$ follows from the assumption that $h$ has zero polynomial part, see \cite[Corollary 4.6(iii)]{Luger_2019}.
    In conclusion, $f\in\text{\emph{PRP}}_2$. Since the representing measure of $h$ is extremal, the same argument used in Section~\ref{section: first_counterexample} (where $\eta=1/4$) shows that $f$ is an extreme point in $\text{\emph{PRP}}_2$.

    Now consider the function
    \begin{equation}\label{eq: F}
        F:=\frac{f-1}{f+1}.
    \end{equation}
    This function possesses a few important properties. Firstly, since $h$ is real rational and therefore is real a.e on $\R^2$, there is for a.e $(\zeta_1,\zeta_2)\in\mathbb{T}^2$ a real number $y=- h((\gamma(\zeta_1),\gamma(\zeta_2)))/\eta$ such that
    \begin{equation*}
        \abs{F(\zeta)} = \bigg| \frac{f(\zeta)-1}{f(\zeta)+1} \bigg| = \bigg| \frac{iy-1}{iy+1} \bigg| = 1.
    \end{equation*}
    Secondly, $\abs{F} < 1$ on $\D^2$, because $\re f > 0$ and therefore
    \begin{equation*}
        \abs{f+1}^2 = (\re f +1)^2+ (\im f)^2 >  (\re f -1)^2+ (\im f)^2 = \abs{f-1}^2
    \end{equation*}
    on $\D^2$. Finally, $F$ is rational. Together, this makes $F$ a rational inner function. By a standard result on rational inner functions \cite{rudinPolydiscs, RudinStout}, we can write $F=z^{k_1}w^{k_2}\tilde{r}/r$ for nonnegative integers $k_1$ and $k_2$, a stable polynomial $r$, and its reflection $\tilde{r}$. We may assume that $r$ and its reflection are coprime, so that $r$ is scattering stable. Since $F(0,\cdot)$ and $F(\cdot,0)$ do not vanish everywhere, we have $k_1=k_2=0$, and hence $F=\tilde{r}/r$.
    
    Combining \eqref{eq: F} and $F=\tilde{r}/r$ yields
    \begin{equation}\label{eq: f}
        f=\frac{r+\tilde{r}}{r-\tilde{r}}.
    \end{equation}
    With this, the conditions in Conjecture~\ref{conj: Knese} are met; indeed, $r$ is scattering stable and $f$ is of the correct form and an extreme point in $\text{\emph{PRP}}_2$. The conjecture thus asserts that $r-\tilde{r}$ is irreducible. By assumption of any of the conditions in Proposition~\ref{prop: reducibility}, $q$ is reducible, and so we are done once we show that $r-\tilde{r}$ is reducible if and only if $q$ is reducible. Denote the bidegree of $h$ as $(m,2)$, $m\geq 1$, and write $f =-\frac{i}{\eta}\frac{p(\gamma,\gamma)}{q(\gamma,\gamma)}=\frac{P}{Q}$ with
     \begin{align*}
        P(z,w) & := -\frac{i}{\eta} (1-z)^m(1-w)^2p\bigg(i\frac{1+z}{1-z},i\frac{1+w}{1-w}\bigg) \quad\text{and}\quad \\
        Q(z,w) & := (1-z)^m(1-w)^2 q\bigg(i\frac{1+z}{1-z},i\frac{1+w}{1-w}\bigg).
    \end{align*}
    Since the bidegree of $h$ and therefore of $q$ is $(m,2)$, $P$ and $Q$ are coprime polynomials in $\C[z,w]$. Since $r+\tilde{r}$ and $r-\tilde{r}$ from \eqref{eq: f} also are coprime, it follows that $Q $ is a nonzero constant multiple of $ r-\tilde{r}$. Since $Q$ is reducible if and only if $q$ is reducible, the proof concludes.
\end{proof}

\begin{example}
    An infinite family of counterexamples to Conjecture~\ref{conj: Knese} is generated from $\varphi = \mathcal{J}^{\circ k}$, $k\geq 1$. The $k$'th member of this family has bidegree $(2^k,2)$, and $U$ from Proposition~\ref{prop: reducibility} is given by $\mathcal{J}^{\circ (k-1)}$. For example, with $k=1$ we recover the bidegree $(2,2)$ function \eqref{eq: final example RR function} considered in Section~\ref{section: extremal_example}. It has a reducible denominator and therefore produces a counterexample to the irreducibility aspect of Conjecture~\ref{conj: Knese}, as shown previously. Its measure is recovered from Remark~\ref{rem: general extremal measure} with $U(z)=z$. We obtain a higher degree example by letting $k=2$. Then
    \begin{equation*}
        \varphi = \mathcal{J}\circ\mathcal{J}=z-\frac{1}{z} - \frac{1}{z-\frac{1}{z}} = z-\frac{1}{z} -\frac{1/2}{z-1}-\frac{1/2}{z+1},
    \end{equation*}
    which yields the bidegree $(4,2)$ real rational Herglotz-Nevanlinna function
    \begin{equation}\label{eq: more complicated example}
        h(z,w)=-\frac{z(z^2-1)w}{(z^2+zw-1)((z^2-1)w-z)}.
    \end{equation}
    Thus, as  discussed in Remark~\ref{rem: general extremal measure}, the extremality of the measure cannot be inferred from the irreducibility of its support, and there are \HNf s with the same denominator as the factors in \eqref{eq: more complicated example}. Indeed,
    \begin{equation*}
        -\frac{1}{z-\frac{1}{z}+w} = -\frac{z}{z^2+zw-1} \quad\text{and}\quad -\frac{1}{w-\frac{1}{z-\frac{1}{z}}} = -\frac{z^2-1}{(z^2-1)w-z},
    \end{equation*}
    are \HNf s as they are compositions of \HNf s.
\end{example}

\appendix 

\section{\bfseries Parametrization of the real zero set}\label{app: A}

\begin{proof}[Proof of Lemma~\ref{lemma: parametrization of zero set}]
    If $m=0$ then $h$, and hence $q$, is independent of $z$. As a polynomial in $w$, $q$ then has $n$ simple zeros $r_1,\dots,r_n$ and the conclusion follows by taking $S=\varnothing$ and $g_j\equiv r_j$. Thus assume $m\geq 1$. The idea now is to first show that these real-analytic functions $g_1,\dots,g_n$ exist and satisfy
    \begin{equation}\label{eq: 1st inclusion}
        \Z(q)\cap ((\R\setminus S)\times \R) \subset { \bigcup_{j=1}^n \{(t,g_j(t)):\,t\in \R\setminus S \}}.
    \end{equation}
    Then we show that any point in $\Z(q)$ on $S\times\R$ is a limit point of a sequence in this union of graphs, which yields
    \begin{equation}\label{eq: 2nd inclusion}
        \Z(q)\cap \Rtwo \subset \overline{ { \bigcup_{j=1}^n \{(t,g_j(t)):\,t\in \R\setminus S \}}}.
    \end{equation}
    The reverse inclusion is immediate by definition of the $g_j$'s and the fact that $\Z(q)\cap \Rtwo$ is closed.

    Fix $\tau\in\R\setminus S$. By Lemma~\ref{lemma: properties of RR functions}(iii), there are $n$ distinct simple real zeros $r_1,\dots,r_n$ of $q(\tau,\cdot)$. In particular, they satisfy $\frac{\partial}{\partial w} q(\tau,r_j)\neq 0$. The implicit function theorem thus applies and yields an open interval $I_\tau\ni\tau$ and unique real-analytic functions $g_j^\tau:I_\tau\to\R$, $j=1,\dots,n$, such that each of them satisfies $g_j^\tau(\tau)=r_j$ and $q(t,g_j^\tau(t))=0$ for every $t\in I_\tau$. In this way, $\Z(q)$ is parametrized by
    \begin{equation*}
        s=g_1^\tau(t),\quad \dots,\quad s=g_n^\tau(t)
    \end{equation*}
    on $I_\tau\times U$, where $U$ is a union of neighbourhoods of the points $r_j$. In fact, by shrinking $I_\tau$ such that $I_\tau\subset\R\setminus S$, the above extends to a parametrization of $\Z(q)$ on $I_\tau\times\R$, because $q(t,u)=0$ then has exactly $n$ distinct solutions for $(t,u)\in I_\tau\times\R$, which are $(t,g_1(t)),\dots \, (t,g_n(t))$. By ordering the locally parametrizing functions on each such interval $I_\tau$, we obtain real-analytic functions $g_j:\R\setminus S \to \R$, $j=1,\dots, n$ such that $\Z(q)$ is parametrized by
    \begin{equation}\label{eq: global param of Zq}
         s=g_1(t),\quad \dots,\quad s=g_n(t)
    \end{equation}
    on $(\R\setminus S)\times\R$. Inclusion \eqref{eq: 1st inclusion} is thus obtained.

    Now consider any point $(\tau,\sigma)\in\Z(q)\cap\Rtwo$ such that $\tau\in S$. Define
    \begin{equation*}
        r(z,w):=q(z+\tau,w+\sigma).
    \end{equation*}
    Then $r(0,0)=0$ and, by assumption that $\tau-z$ is not a factor of $q$, $r(0,w)$ is not identically zero. Moreover, $r$ is nonzero on $\uppertwo$, and therefore also on $(-\upper)^2$, because its coefficients are real. Thus the intersection $\Z(r)\cap(\uppertwo \cup (-\upper)^2)$ is empty. Applying a version of \cite[Lemma 2.10]{lvl_curves_RIFs} where the roles of the two coordinates are interchanged yields a local parametrization of $\Z(r)$ near $(0,0)$ by
    \begin{equation*}
        w = \Psi_1(z), \quad \dots, \quad w = \Psi_L(z),
    \end{equation*}
    for some integer $0<L\leq n$, where each $\Psi_k$ is analytic. Moreover, each $\Psi_k$ is real on $\R$ near $0$; indeed, choose any $\Psi=\Psi_k$ and note that for $z\in\C$ sufficiently close to $0$, $(z,\Psi(z))\in\Z(r)$. Because $\Z(r)\cap(\uppertwo \cup (-\upper)^2)=\varnothing$, we have $\im\Psi\leq 0$ whenever $\im z\geq 0$, while $\im\Psi(z)\geq 0$ whenever $\im z\leq0$. Continuity across $\R$ then gives $\Psi(t)\in\R$ for $t\in\R$ near $0$. 
    
    The above gives a parametrization
    \begin{equation}\label{eq: 2nd parametrized Zq}
        s = \sigma + \Psi_1(t-\tau), \quad \dots,\quad s = \sigma + \Psi_L(t-\tau)
    \end{equation}
    of $\Z(q)\cap\Rtwo$ locally near $(\tau,\sigma)$, where each $\Psi_k$ is real-analytic. Since $q(\tau,\sigma)=0$, at least one of $\Psi_k$ satisfies $\Psi_k(0)=0$, call it $\Psi$. Recall that $S$ is finite. Therefore, there is a one-sided punctured neighbourhood $N$ of $\tau$ contained in $\R\setminus S$. On $N$, the function $t\mapsto \sigma + \Psi(t-\tau)$ agrees with a unique parametrizing function $g$ from the parametrization \eqref{eq: global param of Zq}. Thus
    \begin{equation*}
    g(t) = \sigma + \Psi(t-\tau)
    \end{equation*}
    throughout this neighbourhood, and consequently,
    \begin{equation*}
        (t,g(t)) = (t, \sigma + \Psi(t-\tau)) \to (\tau,\sigma)
    \end{equation*}
    as $t\to \tau$ within this neighbourhood. Hence, the point $(\tau,\sigma)$ appears as the limit point of the set parametrized by \eqref{eq: global param of Zq}. Inclusion \eqref{eq: 2nd inclusion} has thus been obtained, and the proof concludes.
\end{proof}

\section{\bfseries The representing measure structure formula}\label{app: B}

\begin{lemma}\label{lemma: poisson integral of its boundary values}
    Let $h=p/q$ be a real rational Herglotz-Nevanlinna function of two variables such that $q$ and $\tau-z$ are coprime for all $\tau\in\R$. Assume its polynomial part is constant. Fix $w=w_0\in\upper$. Then $z\mapsto \im h(z,w_0)$ is the Poisson integral of its boundary values on $\R$.
\end{lemma}
\begin{proof}
    The polynomial part of $h$ is constant, and hence the same applies to the sliced function $h(\cdot,w_0)$. As such, in the variable $z$, the degree of the denominator is greater than or equal to the degree of the numerator. Euclidean division and then partial fraction decomposition allow us to write $z\mapsto h(z,w_0)$ as a sum of a constant and fractions of the form $c/(d-z)^k$ for some integer $k\geq 1$ and constants $c,d\in\C$. It follows from Lemma~\ref{lemma: properties of RR functions}(i) that the denominator of $z\mapsto h(z,w_0)$ is nonzero on $\R$, in addition to $\upper$, and hence $\im d <0$.  Consequently, $z\mapsto \im h(z,w_0)$ is a harmonic function on $\upper$ which is bounded and continuous on $\overline{\upper}$. In particular, it is the Poisson integral of its boundary values on $\R$; see e.g. \cite[Theorem 7.5]{axler}.
\end{proof}

\begin{proof}[Proof of Theorem~\ref{thm: structure formula representing measures}]
    Write $\mu=\nu+\sigma$, where $\nu$ and $\sigma$ are the unique representing Nevanlinna measures of $\sum_{k=1}^l \frac{c_{z}(\tau_k)}{\tau_k-z}$ and $\tilde{h}$, respectively. If $q$ is coprime with $\tau-z$ for all $\tau\in\R$ then the sum is empty, in which case $\nu\equiv 0$ and $\tilde{h}=h$. Define     
    \begin{equation*}
          \nu' := \pi\sum_{k=1}^lc_{z}(\tau_k)(\delta_{\tau_k}\times\lambda|_{\R\setminus S})
    \end{equation*}
    and
    \begin{equation*}
         \sigma' :=  \pi \sum_{j=1}^n (\gamma_j)_*(\alpha_j\lambda|_{\R\setminus S}).
    \end{equation*}
    We now claim that $\nu=\nu'$ and $\sigma=\sigma'$. The formula for $\nu$ is confirmed through direct computation; indeed, by standard residue calculus we have
    \begin{equation*}
        \int_{\R} K_2((z,w),(t_1,t_2))\,dt_2 = \pi K_1(z,t_1),
    \end{equation*}
    see \cite[Example 3.5]{Luger_2019}, and hence
    \begin{equation*}
       \frac{1}{\pi^2} \int_{\Rtwo} K_2((z,w),(t_1,t_2))\,d\nu' = \sum_{k=1}^lc_{z} \int_{\R\setminus S}  K_1(z,t_1)\,\delta_{\tau_k} = \sum_{k=1}^lc_{z}\bigg(\frac{1}{\tau_k-z}-\frac{\tau_k}{\tau_k^2+1} \bigg).
    \end{equation*}
    It follows by uniqueness of $\nu$ that $\nu=\nu'$, see \cite[Corollary 4.7]{Luger_2019}.
    
    It remains to show that $\sigma=\sigma'$. If $\tilde{h}$ has degree zero in $z$, then $\tilde{h}$ is a one-variable real rational \HNf{} of $w$. By the same computation as for $\nu$ above we find that 
    
    Note that the polynomial part of $\tilde{h}$ is constant just like the polynomial part of $h$, but it does not vanish necessarily. For convenience, we allow a slight abuse of notation and write the reduced function $\tilde{h}=\frac{\tilde{p}}{\tilde{q}}$ as $h=\frac{p}{q}$ throughout the rest of the proof.
    
    Let $w_0\in\upper$ be arbitrary and set $\Phi_{w_0}(z):= \im h(z,w_0)$. By Lemma~\ref{lemma: poisson integral of its boundary values}, $\Phi_{w_0}$ is the Poisson integral of its boundary values:
    \begin{equation}\label{eq: 2nd eq}
        \Phi_{w_0}(z) = \frac{1}{\pi}\int_{\R} \Phi_{w_0}(t)\mathcal{P}_1(z,t)\,dt, \quad z\in\upper.
    \end{equation}
    Recall the finite set $S$ from Lemma~\ref{lemma: properties of RR functions}(iii). Let $t_0\in\R\setminus S$ be arbitrary and set $\Psi_{t_0}(w):= h(t_0,w)$. Because the denominator $q$ of $h$ is coprime with $t_0-z$, the sliced function $\Psi_{t_0} =: \frac{p_{t_0}}{q_{t_0}}$ is a real rational Herglotz-Nevanlinna function of one variable (see Lemma~\ref{lemma: properties of RR functions}(i)) such that $q_{t_0}=q(t_0,\cdot)$ (see Lemma~\ref{lemma: properties of RR functions}(iii)). We therefore also have $p_{t_0}=p(t_0,\cdot)$. In accordance with Section~\ref{section: onevar_prelims}, the representing measure of $\Psi_{t_0}$ is given by
    \begin{equation*}
        \pi\sum_{j=1}^n\bigg|\frac{p_{t_0}(r_j)}{q_{t_0}'(r_j)}\bigg|\delta_{r_j},
    \end{equation*}
    where $r_1,\dots,r_n$ are the $n$ simple real roots of $q_{t_0}$. Since $q_{t_0}=q(t_0,\cdot)$, the denominator of $\Psi_{t_0}$ retains the degree of $q$ in $w$ and hence the linear coefficient of $\Psi_{t_0}$ is zero because those of $h$ are. By Corollary~\ref{cor: main cor}(ii) we thus have
    \begin{equation}\label{eq: im psi_t0}
        \im \Psi_{t_0}(w)  = \int_{\R}\mathcal{P}_1(w,t) \sum_{j=1}^n\bigg|\frac{p_{t_0}(r_j)}{q_{t_0}'(r_j)}\bigg|\,d\delta_{r_j}(t) 
         = \sum_{j=1}^n\bigg|\frac{p_{t_0}(r_j)}{q_{t_0}'(r_j)}\bigg| \mathcal{P}_1(w,r_j), \quad w\in\upper. 
    \end{equation}
    Note that $q_{t_0}'(r_j)=\frac{\partial q}{\partial w}(t_0,r_j)$. Moreover, the real-analytic functions
    \begin{equation*}
        g_1,\dots,g_n:\R\setminus S \to\R
    \end{equation*}
    from Lemma~\ref{lemma: parametrization of zero set} satisfy $r_j = g_j(t_0)$ for each $j$. In addition, at $(t_0,w_0)\in\R\times\upper$, we have $\im\Psi_{t_0}(w_0)=\Phi_{w_0}(t_0)$. By this, we may rewrite \eqref{eq: im psi_t0} as
    \begin{equation}\label{eq: 3rd eq}
        \Phi_{w_0}(t_0) =  \sum_{j=1}^n \alpha_j(t_0)\mathcal{P}_1({w_0},g_j(t_0)).
    \end{equation}
    Since $t_0\in\R\setminus S$ and $w_0\in\upper$ were arbitrary, and $S$ is finite, we can substitute \eqref{eq: 3rd eq} into \eqref{eq: 2nd eq} and thus obtain
    \begin{equation*}
        \im h(z,w) = \frac{1}{\pi} \sum_{j=1}^n \int_{\R\setminus S} \Ptwo((z,w),(t,g_j(t))) \alpha_j(t)\,dt = \frac{1}{\pi^2}\int_{\Rtwo}\Ptwo((z,w),t)\,d\sigma'(t), \quad (z,w)\in\uppertwo.
    \end{equation*}
    Evaluating the above at $(z,w)=(i,i)$ shows that $\sigma'$ satisfies the growth condition \eqref{eq: growth condition Nevanlinna measure}, and hence $\int_{\Rtwo}K_2((z,w),t)\,d\sigma'(t)$ is well-defined for $(z,w)\in\uppertwo$. Similar to before, it now follows by uniqueness of $\sigma$ that $\sigma=\sigma'$, see \cite[Corollary 4.7]{Luger_2019}. 
\end{proof}

\section{\bfseries An auxiliary lemma}\label{app: X}
On the topic of extremal measures, we need the following lemma. It is an application of the Fatou theorem for Poisson integrals of functions in $L^\infty(\R,\R;\lambda)$.

\begin{lemma}\label{lemma: H is identically zero}
    If $u\in L^\infty(\R,\R;\lambda)$ and
    \begin{equation*}
        \int_{\R} \frac{u(t)}{(t-z)^2}\,dt=0, \quad z\in\upper,
    \end{equation*}
    then $u$ is a.e constant.
\end{lemma}
\begin{proof}
    Recall the one-dimensional Poisson kernel
    \begin{equation*}
        \mathcal{P}_1(x+iy,t)=\frac{y}{(t-x)^2+y^2},
    \end{equation*}
    and notice that 
    \begin{equation*}
        \frac{1}{(t-z)^2} = \frac{(t-x)^2-y^2}{((t-x)^2+y^2)^2} + i\frac{2y(t-x)}{((t-x)^2+y^2)^2} = \partial_y\mathcal{P}_1(x+iy,t) + i \partial_x\mathcal{P}_1(x+iy,t).
    \end{equation*}
    By considering the real and imaginary parts of the integral, we thus obtain the two equations
    \begin{equation}\label{eq: re and im integrals}
        \int_{\R} u(t) \partial_y\mathcal{P}_1(x+iy,t) \,dt = 0 \quad\text{and}\quad \int_{\R} u(t)\partial_x\mathcal{P}_1(x+iy,t) \,dt = 0, \quad x+iy\in\upper.
    \end{equation}
    The function
    \begin{equation*}
        (x,y)\mapsto \frac{1}{\pi}\int_{\R} u(t) \mathcal{P}_1(x+iy,t) \,dt
    \end{equation*}
    is the Poisson integral of $u\in L^\infty(\R,\R;\lambda)$, which by the Fatou theorem for Poisson integrals of functions in $L^\infty(\R,\R;\lambda)$ has nontangential limit $u$ for a.e point in $\R$, see e.g. \cite[Theorem 7.28]{axler}. Interchanging integration and differentiation in \eqref{eq: re and im integrals} shows that this Poisson integral is constant, which means that $u$ is constant a.e.
\end{proof}

\section*{AI disclosure}

OpenAI's model \emph{GPT-5.6 sol} was used in the final stages of the preparation of the original manuscript for identifying potential mathematical or logical errors. During this process, the model pointed out that the function \eqref{eq: final example RR function}, for which it had been shown in the text that the representing measure is extremal, could be transformed to the bidisc setting and by this would also serve as a counterexample to  Knese's conjecture. This hint led to Section~\ref{section: first_counterexample}, and subsequently motivated the development of Section~\ref{section: reducibility}. The mathematical arguments and proofs in these subsections were developed and verified by the authors. No other significant contributions were made by AI.

\bibliographystyle{amsplain}
\bibliography{bibliographyFile}

\end{document}